\documentclass[a4paper, 10pt]{amsart}

\usepackage{amsmath, amsfonts, amsthm, amssymb, amscd, mathrsfs, tikz-cd, hyperref}

\usepackage[utf8]{inputenc}
\usepackage[a4paper,margin=1in]{geometry}
\usepackage{latexsym,amsmath,amsfonts,amscd,amssymb, hyperref}
\usepackage{graphicx}

\makeindex

\title{The Moduli Space of Generalized Quivers}
\author{Artur de Araujo}
\date{}

\newcommand{\GL}{\mathrm{GL}}
\newcommand{\Or}{\mathrm{O}}
\newcommand{\Sp}{\mathrm{Sp}}

\newcommand{\Hom}{\mathrm{Hom}}
\newcommand{\End}{\mathrm{End}}

\newcommand{\tr}{\mathrm{tr}}

\newcommand{\spec}{\mathrm{Spec}\phantom{.}}
\newcommand{\proj}{\mathrm{Proj}\phantom{.}}
\newcommand{\sslash}{\mathbin{/\mkern-6mu/}}

\newcommand{\Rep}{\mathrm{Rep}}
\newcommand{\Ad}{\mathrm{Ad}}
\newcommand{\ad}{\mathrm{ad}}
\newcommand{\Stab}{\mathrm{Stab}}

\newcommand{\ta}{t(\alpha)}
\newcommand{\ha}{h(\alpha)}
\newcommand{\sa}{\sigma (\alpha)}

\theoremstyle{plain}  
\newtheorem{theorem}{Theorem}[section]
\newtheorem*{theorem*}{Theorem}
\newtheorem{corollary}[theorem]{Corollary}
\newtheorem{lemma}[theorem]{Lemma}
\newtheorem{proposition}[theorem]{Proposition}

\theoremstyle{definition}
\newtheorem{definition}[theorem]{Definition}

\theoremstyle{remark}

\newtheorem*{notation*}{Notation}
\newtheorem{remark}[theorem]{Remark}

\newtheorem*{claim*}{Claim}

\numberwithin{equation}{section}

\begin{document}

\begin{abstract}
We construct the moduli space of finite dimensional representations of generalized quivers for arbitrary connected complex reductive groups using Geometric Invariant Theory as well as Symplectic reduction methods. We explicit characterize stability and instability for generalized quivers in terms of Jordan-H\"older and Harder Narasimhan objects, reproducing well-known results for classical case of quiver representations. We define and study the Hesselink and Morse stratifications on the parameter space for representations, and bootstrap them to an inductive formula for the equivariant Poincaré Polynomial of the moduli spaces of representations. We work out explicitly the case of supermixed quivers, showing that it can be characterized in terms of slope conditions, and that it produces stability conditions different from the ones in the literature. Finally, we resolve the induction of Poincar\'e polinomials for a particular family of orthogonal representations.
\end{abstract}

\maketitle

\section*{introduction}

Generalized quivers unify a diversity of objects in current use by stripping these particular instances to a bare essential: that in each case we're dealing with adjoint actions of a group on some Lie algebra. In representation-theoretic terms, they make headway into a detailed understanding of linear representations of reductive groups. In geometric terms, they are likely to serve as a basis for the extension of several results in the theory of quivers, including detailed explicit results on gauge-theoretical moduli spaces. 

Just as their name indicates, generalized quivers had their immediate motivation in the theory of quiver representations. In fact, classical quivers, along with quivers with additional symmetries constitute main examples. In this paper we apply general machinery in the theory of quuotients to define stability conditions for generalized quivers, characterize the correrspoding stratifications in the space of representations, and deduce an inductive formula for the equivariant cohomology of the resulting moduli spaces. When spacializing these results to the classical examples we just mentioned, we will recover many of the main result about them. In that sense, generalized quivers make headway into a Lie theoretic framework for those problems.

Our main definition is the following.
\begin{definition}
Let $G$ be a reductive group, $\mathfrak{g}$ its Lie agebra.
\begin{enumerate}
\item A \emph{generalized $G$-quiver $\tilde{Q}$ with dimension vector} is a pair $(R,\mathrm{Rep}(\tilde{Q}))$ where $R$ is a closed reductive subgroup of $G$, and $\mathrm{Rep}(\tilde{Q})$ a finite-dimensional representation of $R$ (the \emph{representation space}.) We require the irreducible factors of the representation also to be irreducible factors of $\mathfrak{g}$ as an $\mathrm{Ad} R$-module, and the trivial representation to not occur.

\item A generalized quiver \emph{of type Z} is a generalized quiver for which $R$ can be realized as a centralizer in $G$ of a some closed abelian reductive subgroup.

\item A \emph{representation of $\tilde{Q}$} is a vector $\varphi \in \mathrm{Rep}(\tilde{Q})$.
\end{enumerate}
\end{definition}

This definition is essentially due to Derksen-Weyman \cite{dw}, though they require generalized quivers always to be of type Z. (This has the important consequence that $R$ is then a Levi subgroup.) Note that the definition above applies equally well to real or to complex Lie groups, as well as to reductive algebraic groups over some field $k$ (in fact, Derksen-Weyman's original setting.) However, apart from an incidental appearance of unitary groups, in this paper we will restrict to \emph{affine, linearly reductive complex groups}.

Our main problem is the following: classify representantions of generalized quivers up to isomorphism; in other words, construct and characterize the quotient $\Rep (\tilde{Q})\sslash R$. It is clear from the start that this problem fits in the general framework of both affine GIT and symplectic geometry, and it is a standard, but fundamental fact that both of these coincide; this `duality' permeates our treatment. It also turns out that the study of generalized quivers is (not surprisingly) deeply intertwined with Lie theory, and such intertwining makes this topic fairly interesting in the interactions that it brings to light. On the other hand, while the background required is rather standard, it is spread through different areas. Quivers have bearings in areas were this material is not necessarily well know in its entirety, so we have been careful to spread out a summary of this background in each relevant section. Our comments are necessarily very brief, so cannot possibly serve as an introduction to the subject. It does however serve to set notation, and we hope that it is systematic enough to serve as a guide through the paper. Needless to say, a nuanced understanding of this rich topic is only to be gained through a study of the relevant sources, so we've tried to point them out in situ.

Our main result is an inductive formula for the equivariant Poincar\'e polynomial of the locus of semistable representations. Recall that if the action of the group is free, this is the Poincar\'e polynomial of the moduli space. This formula is deduced by methods inspired by work of Atiyah and Bott on the moduli space of connections over a Riemann surface. They used the square-norm of a moment map to define a stratification of the space of connections, and extract the equivariant polynomials from the Thom-Gysin sequence for that stratification. This work was later adapted to the projective case by Kirwan \cite{kirwan}, which also proved that the stratification coincided with an algebraic one which is defined in terms of the instability of points on the variety. We will use an extension of this theory to linear actions on affine spaces. We find that after retracting the strata, the Poincar\'e polynomial is expressed in terms of representations spaces of lower rank, so that the whole process is in fact computable. We provide an example to show this.

\subsection{Acknowledgments}
I want to thank my adviser, Peter Gothen, for his expertise and advice, essential in more than one aspect for this work. I also want to thank Oscar Garcia-Prada and Luis Alvarez-Consul, who received me more than one at ICMAT Madrid. Andre Oliveira spent an undue ammount of time listening to my thoughts and questions, and also deserves a mention. Last but not least, Ana Peon-Nieto, with whom I had various discussions.

I have been supported by Fundação para a Ciência e a Tecnologia, IP (FCT) under the grant SFRH/BD/89423/2012, by Funda\c c\~ ao Calouste Gulbenkian under the program Est\'imulo \`a Investigação, partially supported by CMUP (UID/MAT/00144/2013) and the project PTDC/MAT-GEO/2823/2014 funded by FCT (Portugal) with national funds and by CMUP (UID/MAT/00144/2013), which is funded by FCT (Portugal) with national (MEC) and European structural funds through the programs FEDER, under the partnership agreement PT2020.

\section{Preliminaries}\label{preliminaries}

\subsection{Parabolic and Levi subgroups}\label{parabolic}

In the theory of quotients, parabolic and Levi subgroups play a prominent role. We briefly review their definition, as well as some facts that we'll need later on. Good references are \cite{ggm} \cite{springer}.

We will consider the case of complex Lie groups because this is precisely the setting we'll use below; there are ready translations to the setting of complex algebraic groups . Note that this is in fact a formal distinction, since 

\subsubsection{The definitions}
In this paper, we will take $G$ to be a connected, lineraly reductive linear algebraic group over the complex numbers. This is in fact equivalent to $G$ being reductive in two other common senses, which we make explicit for reference. First, we say a complex Lie group $G$ is reductive if it is the complexification of any of its maximal compacts $K$, which we'll take as fixed; every such connected $G$ admits a unique structure as an algebraic group, and it is linearly reductive as well as faithfully representable. On the other hand, every complex algebraic group with these two latter properties is not only smooth, but a reductive complex Lie group. Further, such algebraic groups are in fact linear.

Fix then a group complex reductive $G$ with fixed maximal compact $K$, and denote by $\mathfrak{g}$ and $\mathfrak{k}$ the corresponding Lie algebras. If $\mathfrak{z}$ is the centre of $\mathfrak{g}$ and $T$ a maximal torus of $K$, there is a choice of Cartan subalgebra $\mathfrak{h}$ such that $\mathfrak{z}\oplus \mathfrak{h}=\mathfrak{t}^\mathbb{C}$, where $\mathfrak{t}=\mathrm{Lie}\phantom{.}T$. Let $\Delta$ be a choice of simple roots for the Cartan decomposition of $\mathfrak{g}$ with respect to $\mathfrak{h}$. For any subset $A=\{ \alpha_{i_1},...,\alpha_{i_s}\} \subset \Delta$, define
\begin{equation*}
D_A=\{ \alpha \in R \phantom{.} | \phantom{.} \alpha =\sum m_j\alpha_j, \textrm{ $m_{i_t}\geq 0$ for $1\leq t\leq s$} \}
\end{equation*}

The \emph{parabolic subalgebra} associated to $A$ is 
\begin{equation*}
\mathfrak{p}_A=\mathfrak{z}\oplus \mathfrak{h}\oplus \bigoplus_{\alpha\in D_A} \mathfrak{g}_\alpha
\end{equation*}
This subalgebra determines a subgroup $P_A$ of $G$, the \emph{standard parabolic subgroup} determined by $A$.

\begin{definition}
A \emph{parabolic subgroup} of $G$ is a subgroup $P$ conjugate to some standard parabolic subgroup $P_A$. A \emph{Levi subgroup} of $P$ is a maximal reductive subgroup of $P$.
\end{definition}

For the case of standard parabolic subgrooups there is a ``natural'' choice of a Levi subgroup.  Let $D_A^0\subset D_A$ be the set of roots with $m_j=0$ for $\alpha_j \in A$. Then, 
\begin{equation*}
\mathfrak{l}_A=\mathfrak{z}\oplus \mathfrak{h}\oplus \bigoplus_{\alpha\in D^0_A} \mathfrak{g}_\alpha
\end{equation*}
is the standard Levi subalgebra of $\mathfrak{p}_A$, and the connected subgroup $L_A$ determined by $\mathfrak{l}_A$ is a Levi subgroup of $P_A$.

We will need an important construction of parabolic subgroups. For any $\beta \in \mathfrak{k}$, let
\begin{align*}
\mathfrak{p}(\beta) &:= \{ x\in \mathfrak{g} | \mathrm{Ad}(\exp{it\beta})x \textrm{\phantom{.}remains bounded as $t\to 0$}\} \\
\mathfrak{l}(\beta) &:= \{  x\in \mathfrak{g} | [\beta, x]=0\}
\end{align*}
Let $P(\beta)$ and $L(\beta)$ be the corresponding closed subgroups of $G$ and their Levis. The following is in \cite{ggm}:
\begin{lemma}\label{oneparparabolic}
The sugroups $P(\beta)$ and $L(\beta)$ are a parabolic, resp. Levi subgroup of $G$. Conversely, for any parabolic subgroup of $P$ there is a $\beta \in \mathfrak{k}$ such that $P=P(\beta)$.
\end{lemma}

\subsubsection{Dominant elements}
The description above is in fact a statement about parabolic subgroups and their dominant characters. Note that the characters of $\mathfrak{p}_A$ are in bijection with elements of $\mathfrak{z}^*\oplus \mathfrak{c}_A^*$.
\begin{definition}
An \emph{dominant  character} of $\mathfrak{p}_A$ is an element of the form $\chi=z+\sum_{\delta\in A}n_\delta \lambda_\delta$ where $n_\delta$ is a non-positive real number. It is strictly dominant if those integers are actually strictly negative. The \emph{dominant weights} are the correspoding elements of $\mathfrak{z}\oplus \mathfrak{c}_A$ through the chosen invariant form on the latter.
\end{definition}
We have \cite{ggm}:
\begin{lemma}
Given a dominant weight $\beta$ of $P$, $P\subset P(\beta)$, equality holding if and only if $\beta$ is strictly anti-dominant.
\end{lemma}

\subsubsection{Interpretation in terms of flags}
The construction above is key to understanding an interpretation of parabolic subgroups in terms of flags. This interpretation will be important for us below in establishing a connection between the Lie-theoretic framework, and the theory of classical quivers.

Fix a faithful representation $\rho: K\to U(V)$ (complexifying, also $\rho: G\to GL(V)$.) This induces an isomorphism $\mathfrak{k}\simeq \mathfrak{k}^*$, which we shall use implicitly. Because $\rho$ is a unitary representation, the image $\rho_* \beta$ of any element $\beta \in i\mathfrak{k}$ in $\GL (V)$ is Hermitian, so it diagonalizes with real eigenvalues $\lambda_1<...<\lambda_r$, and induces a filtration
\begin{equation*}
0 \neq V^1\varsubsetneq ... \varsubsetneq V^r =V
\end{equation*}
where $V^k=\bigoplus_{i\leq k}V_{\lambda_k}$ is the sum of all eigenspaces $V_{\lambda_i}$ with $i\leq k$.
Let $St(\beta)\subset \GL (V)$ be the subgroup stabilizing this flag; in other words, $g\in St(\beta)$ if and only if $g\cdot V^r\subset V^r$ for all $r$.
\begin{proposition}
For any $\beta\in i\mathfrak{k}$, we have $\rho^{-1}(St(\beta))=P(\beta)$.
\end{proposition}
Using Proposition \ref{oneparparabolic}, it is now easy to characterize parabolic subgroups of $G$ as stabilizers of certain flags. Levi subgroups then correspond to stabilizers of the associated graded vector space, i.e., stabilizers of the decompositions
\begin{equation*}
V=\bigoplus V_{\lambda_i}
\end{equation*}
induced by an element $\beta \in i\mathfrak{k}$.

It is important, however, to keep in mind that this applies to filtrations induced by an element of $i\mathfrak{k}$, and not just any filtration of $V$. For $\GL (V)$, it is true that any flag is so induced. But for an orthogonal or symplectic group, for example, the flags induced in this way are \emph{isotropic}: they satisfy $V^{r-k}=(V^{k})^\perp$ (so that for $k\leq r/2$, the spaces are actually isotropic.)

\subsubsection{Algebraic groups} For algebraic groups, the natural objects are one-parameter subgroups, and not elements of the Lie algebra. Given a one-parameter subgroup (OPS) $\lambda$ of $G$, we can, in fact, define the analogue of the above parabolic and Levi subgroups by
\begin{align*}
P(\lambda)&:=\{ g\in G \left| \phantom{.}\lim_{t\to 0} \mathrm{Ad}(\lambda (t))g \phantom{.}\textrm{exists} \right.\} \\
L(\lambda)&:=\{ g\in G \left| \phantom{.}\lim_{t\to 0} \mathrm{Ad}(\lambda (t))g=g \right.\}
\end{align*}

In fact, the analogue of Proposition \ref{oneparparabolic} holds, that is, this construction yields all parabolic subgroups of $G$ (\cite{springer} Proposition 8.4.5.) But in fact we can relate both constructions, since any one-parameter subgroup of $G$ is conjugate to some other which sends the maximal compact subgroup $U(1)$ into the maximal compact $K$. In other words, each one-parameter subgroup is uniquely determined by an element of $\mathfrak{k}$, since
\begin{proposition}
$\Hom (U(1),K)\otimes_\mathbb{Z} \mathbb{R} \simeq \mathfrak{k}$
\end{proposition}

One sees in this way that the two settings are completely interchangeable.

\section{Quotients}

The setting throughout the section is as follows. Let $V$ be a hermitian vector space, a  finite dimensional complex vector space with a fixed hermitian form $(\cdot , \cdot )$. The anti-symmetrization of the hermitian form is a (real) symplectic form on $V$:
\begin{equation*}
\omega(v,w)=2\mathrm{Im}(v,w)
\end{equation*}
We will further denote the coordinate ring of $V$ by $R:=\mathbb{C}[V]$.

Let $G$ be a connected, linearly reductive linear algebraic group over the complex numbers, with a fixed maximal compact $K$. Suppose $G$ acts on $V$ through a regular representation $\rho:G\to \GL(V)$ which restricts to a unitary representation $K\to \mathrm{U} (V)$; we allow $\rho$ to have a kernel $\Delta$.

\subsection{Geometric Invariant Theory}

\subsubsection{The quotient} Recall from Chpater 1 that the algebraic construction of a quotient involves a choice of a character $\chi: G\to \mathbb{G}_m$ of $G$. In general we must always have $\chi ([G,G])=1$, and we will also require that $\chi (\Delta)=1$.\footnote{We could omit this condition, but it is necessary to ensure the existence of semistable points as we shall see below (cf. also \cite{king}.)} If we denote $F:=\ker \chi$, and $R_{\chi,n}$ the set of elements in $R$ such that $g\cdot r=\chi(g)^n r$ (the \emph{semi-invariants} of weight $n$), we have
\begin{equation}
R^F=\bigoplus R_{\chi, n}
\end{equation}
where $R^F$ is the ring of $F$-invariants. We then define the \emph{GIT quotient} to be
\begin{equation*}
V\sslash_\chi G := \proj R_{\chi, n}
\end{equation*}

We can make sense of this definition as follows. With respect to the action of a reductive group $F$, the ring of invariants $R^F$ is final for the subrings made up of invariant elements of $R$. By a theorem of Hilbert, it is also an affine ring, which means that $\spec R^F$ is an affine variety with the following universal property: every $F$-invariant map $V\to Y$ factors through the natural map $\pi: V\to \spec R^F$ induced by the inclusion $R^F\to R$. In other words, $\spec R^F$ is a categorical quotient. In the sense that $\proj$ factors out the $\mathbb{G}_m$ action, we may then see $V\sslash_\chi G$ in fact as a quotient by $G$. An important fact to notice, however, is that precisely because this definition involves a projective quotient, there is only a rational map $V\dashrightarrow V\sslash_\chi G$, which is not generally regular. This means that in a strict sense, we are finding a quotient only for a (Zarisky) open set. In fact, the map is defined only in the open locus of \emph{semistable points}, where the relevant definitions are as follows.
\begin{definition}
A point $x\in X$ is
\begin{enumerate}
\item \emph{$\chi$-semistable} if there is some semi-invariant which does not vanish at $x$.
\item \emph{$\chi$-polystable} if its orbit is closed in the set $X^{\chi-ss}$ of semistable points.
\item \emph{$\chi$-stable} if it is polystable and it is simple, i.e., its stabilizer is precisely $\Delta$.
\end{enumerate}
\end{definition}
The complement of $X^{ss}$ (i.e. the locus where the map is undefined) is called the \emph{null cone} of $X$. Note that polystable points are necessarily semistable, and in fact, the GIT quotient parametrizes closed orbits in $X^{\chi-ss}$. Thus, this quotient is in fact an orbit space for polystable points. On stable points as defined, the quotient has nice geometric properties, in particular it is geometric (it actually parametrizes orbits) and smooth. We want to remark that in the literature, it is often only required that the quotient of stable points be geometric, not smoooth; the corresponding condition on stabilizers is just that it contain $\Delta$ with finite index.

\begin{remark}
We could also have taken the categorical quotient for $G$, and in fact it is retrieved for the trivial character. This would have the advantage that \emph{every} point would be semistable. However, such a quotient is usually very restrictive, as can already be seen in the simplest examples. The choice of the character goes a great way to remedy this. One interesting thing to note is that, because invariants separate closed orbits, every stable point for the trivial character is stable for \emph{any} character.
\end{remark}

Geometrically speaking, the quotient we just defined amounts to the quotient of $V$ as a quasi-projective variety. With the character $\chi$ we can make $G$ act on the affine ring $R[z]$ by letting $g\cdot z=\chi(g)^{-1}z$. Again by the theorem of Hilbert, $R[z]^G$ is also an affine ring, and it inherits a grading from $R[z]$ according to the powers of $z$. A simple observation is the following: an element $r\otimes z^n$ is $G$-invariant if and only if $r\in R_{\chi,n}$. This implies in particular that $(R[z]^G)_n =R_{\chi, n}$, which means that
\begin{equation*}
V\sslash_\chi G= \proj R[z]^G
\end{equation*}
Now, $R[z]=R\otimes_k k[z]$ is the coordinate ring of the trivial line bundle $L$ over $V$, and the co-action of $G$ on $R[z]$ through the character $\chi$ corresponds to an action of $G$ on $L^{-1}$ by the formula
\begin{equation*}
g\cdot (v,z)=(g\cdot v, \chi(g)^{-1}z)
\end{equation*}
Now, $L$ is the pullback of the anticanonical line bundle $\mathscr{O}(1)$ for any embedding $X\to \mathbb{P}^n$ into some projective space, and so $L^{-1}=\mathscr{O}_V(-1)$ is the blow-up of the corresponding affine cone over $V$ at the origin. Further, $\bigoplus_{n\geq 0} \mathscr{O}(n)$ is the coordinate ring of that cone, so that $\proj R[z]^G$ essentially corresponds to taking the quotient of that affine cone and then projecting down to some projective space. This geometric interpretation explains the so-called  \emph{topological criterion}.
\begin{theorem}
Let $v\in V$ be any point, and $\hat{v}$ be an arbitrary lift of $v$ to the total space of $L^{-1}$. Then,
\begin{enumerate}
\item The point $v$ is $\chi$-semistable if and only if the closure of the orbit $G\cdot \hat{v}$ is disjoint from the zero section;
\item The point $v$ is $\chi$-stable if and only if $G\cdot \hat{v}$ is closed, and its stabilizer is precisely $\Delta$.
\end{enumerate}
\end{theorem}

\subsubsection{The Hilbert-Mumford criterion}
The topological criterion for a GIT quotient shows that we need to consider the existence of certain limit points in the zero section. The Hilbert-Mumford criterion essentially states that we can do that by checking one-dimensional paths generated by elements of $G$. We will here explain this criterion in the affine case.

Let $\lambda$ be a one-parameter subgroup (OPS) of $G$. If $\lambda (t)\cdot x$ does not converge to any point, then clearly for any lift $\hat{x}_0$, the orbit $\lambda \cdot \hat{x}_0$ is disjoint from the zero section. Suppose then that the limit $x_0=\lim_{t\to 0} \lambda (t)\cdot x$ exists. The point $x_0$ is necessarily a fixed point of the action of $\lambda$ (i.e., the action of $\mathbb{C}^*$ through $\lambda$,) so that on $L^{-1}$ this action restricts to the fibre over $x_0$. This action on that fibre is just mulitplication by $\chi(\lambda(t))^{-1}=t^{-\langle \chi, \lambda\rangle}$. In other words, for any lift $\hat{x}_0$ of $x_0$ we have
\begin{equation*}
\lambda (t)\cdot \hat{x}_0=t^{-\langle \chi, \lambda\rangle}\hat{x}_0
\end{equation*}
Clearly, if $\langle \chi, \lambda\rangle$ is strictly negative, then as $t \to 0$ the orbit $\lambda \cdot \hat{x}_0$ adheres to the zero of the fibre. This is also true for any lift $\hat{x}$ of $x$, and so the topological criterion implies that $x$ is unstable. The content of the following theorem is that unstability can always be checked in this way.
\begin{theorem}[\cite{king} 2.5]
Let $v \in V$. Then,
\begin{enumerate}
\item The point $v$ is $\chi$-semistable if and only if $\langle \chi, \lambda \rangle \geq 0$ for all one-parameter subgroups $\lambda$ for which $\lim \lambda(t)\cdot v$ exists.
\item The point $v$ is $\chi$-stable if and only if it is semistable, and the only $\lambda$ for which $\lim \lambda(t)\cdot v$ exists and $\langle \chi, \lambda \rangle=0$ are in $\Delta$.
\end{enumerate}
\end{theorem}

This criterion explains the requirement that $\chi(\Delta)=1$, for otherwise some one-parameter subgroup of $\Delta$ would destabilize every point. It is hard to overstate the importance of this criterion. In many important examples (classical quivers included,) an explicit calculation of this criterion leads to slope-conditions that are very explicit and descriptive.

The number $\langle \chi ,\lambda \rangle$ is the \emph{Hilbert-Mumford pairing}. It is worth noting that, in contrast with the projective case, this pairing is independent of the point $x$. On the other hand, one considers not all one-parameter subgroups, but only those in the set
\begin{equation*}
\chi_*(G,v):=\{\lambda \in \chi_*(G)|\phantom{.}\textrm{$\lim_{t\to 0}\lambda(t)\cdot v$ exists}\}
\end{equation*}
We want to use this pairing to classify the instability of points as in the projective case. This relies on important properties of the pairing in the following proposition.
\begin{proposition}
Let $v\in V$ and $\lambda \in \chi_*(G,v)$ be arbitrary, and denote $v_0=\lim_{t\to 0} \lambda(t)\cdot v$. Then,
\begin{enumerate}
\item $\chi_*(G,g\cdot v)=g\chi_*(G,v)g^{-1}$ for any $g\in G$.

\item $g\lambda g^{-1} \in \chi_*(G,v)$ for any $g\in P(\lambda)$.

\item $\langle \chi, g\lambda g^{-1} \rangle = \langle \chi, \lambda \rangle$ for any $g\in G$.
\end{enumerate}
\end{proposition}

 Let $||\cdot ||$ be a fixed, $G$-invariant norm on the space $\chi_* (G)$ of one-parameter subgroups of $G$. This always exists since, having fixed a maximal torus of $G$, choosing such a norm is equivalent to choosing a norm on $\chi_* (T)$ invariant under the action of the Weyl group, which is finite. (It is clear from this, however, that such a choice is in general far from unique.) Alternatively, such a norm is also equivalent to a choice of a $K$ invariant norm on $\mathfrak{k}$. We will assume that this norm is \emph{integral}, that is, for all $\lambda \in \chi_*(G)$ we have $||\lambda|| \in \mathbb{Z}$, or equivalently that $||\alpha|| \in \mathbb{Z}$ for all integral weights of $\mathfrak{k}$ (we can always use a multiple of the Killing form on $\mathfrak{k}$.)

\begin{definition}
For any $v\in V$, let
\begin{equation}
M^\chi_G(v)=\inf \left\{  m_\chi(\lambda):=\frac{\langle\chi, \lambda\rangle}{||\lambda||} , \lambda \in \chi_*(G,v) \right\}
\end{equation}
Further, let $\Lambda_G(v)\subset \chi_*(G,v)$ be the set of indivisible $\lambda$ with $m_\chi(\lambda)=M_G^\chi(v)$.
\end{definition}
Indivisible here means that $\lambda$ is not a positive power of another one-parameter subgroup; alternatively, since $\lambda \in \chi_*(T)$ for some maximal torus $T\subset G$, and $\chi_*(T)$ is a lattice, indivisibility means $\lambda$ is minimal in the lattice.  Our first goal is to prove the following:
\begin{lemma}
$M_G^\chi(v)$ is a finite number for all $v\in V$, and $\Lambda_G(v)$ is non-empty.
\end{lemma}
\begin{proof}
There is a maximal torus $T$ of $G$ for which $\lambda$ is also a one-parameter subgroup. Through the chosen invariant form, $\chi$ determines an element of $\mathfrak{t}$, which we also denote by $\chi$. The number $m_\chi(\lambda)$ is clearly the component of $\lambda$ along $\chi$. Since the norms of indivisible $\lambda$ are bounded, so is is the collection of $m_\chi(\lambda)$, so that $M_G^\chi(v)$ is finite. It is also easy to see from this that for the action of $T$ there is certainly a minimum, so that $\Lambda_G(x)$ is non-empty.
\end{proof}

\subsubsection{The Hesselink stratification}
We now define the Hesselink stratification for $V$. Note that in the affine case, the instability of a point is completely characterized by the set $\Lambda_G(v)$, since $m_\chi(\lambda)$ is independent of the point $v$ for all $\lambda$. The following result of Kempf shows that this set is contained in some adjoint orbit, just as in the projective case.
\begin{lemma}[Kempf]
Let $v\in V$ be unstable. Then, $\Lambda_c(v)$ is non-empty, and there is a (unique) parabolic $P(v)$ such that $P(v)=P(\lambda)$ for every $\lambda\in \Lambda_c(v)$. Furthermore, for $\lambda,\lambda'\in \Lambda_c(v)$, we have $\Ad (g)\lambda=\lambda'$ if and only if $g\in P(v)$.
\end{lemma}
With this we now define $B$ to be the set of adjoint orbits $[\lambda]$ of one-parameter subgroups for which there is a $v$ such that $\Lambda_G(v)\subset [\lambda]$. This will be the index set for the stratification. We define an (strict) ordering on $B$ by setting $[\lambda]<[\lambda']$ if $m_\chi([\lambda])<m_\chi([\lambda'])$.

\begin{definition}
Let $[\lambda]$ be a conjugacy class of one parameter subgroups of $G$. The \emph{Hesselink stratum indexed by $[\lambda]$} is the set
\begin{equation*}
S_{[\lambda]}:=\{ v\in V^{us} \left| \phantom{.}\Lambda_c(v)\subset [\lambda] \right. \}
\end{equation*}
For each $\lambda' \in [\lambda ]$, the \emph{blade defined by $\lambda'$} is
\begin{equation*}
S_{\lambda'}:=\{ v\in V^{us} \left| \phantom{.}\lambda'\in \Lambda_c(v)\right. \} \subset S_{[\lambda]}
\end{equation*}
\end{definition}
\begin{theorem}[Hoskins \cite{hoskins} 2.16]
The collection $S_{d,[\lambda]}$ is a stratification of $X$, i.e., $X$ is the disjoint union of the sets, and the ordering on $B$ is such that
\begin{equation*}
\overline{S}_{[\lambda]}\subset \bigcup_{[\lambda]\leq [\lambda']} S_{[\lambda']}
\end{equation*}
\end{theorem}

We will need the following complement to Hoskin's results.
\begin{lemma}\label{bladeaffine}
For any $\lambda$, $S_{[\lambda]}\simeq G\times_{P(\lambda)} S_{\lambda}$.
\end{lemma}
\begin{proof}
It is easy to see that there is continuous bijection $\overline{\sigma}:G\times_{P(\beta)}S_\beta \to S_{[\beta]}$ as follows: to start with, the restriction of the action $\sigma: G\times S_\beta \to S_{[\beta]}$ is a surjection since $S_{[\beta]}=GS_\beta$. Now, this map factors through $G\times_{P(\beta)}S_\beta$ because if $(g',y')=(gp^{-1},py)$ (which makes sense since $P_\beta$ stabilizes $S_\beta$,) then clearly $g'y'=gy$. This factorization is the desired continuous bijection $\overline{\sigma}$, and it is surjective because $\sigma$ is so. It is injective, note that $g'y'=gy\iff y=g^{-1}g'y'$ implies that $(g'(g^{-1}g')^-1,g^{-1}g'y')=(g,y)$.

The problem now is to show that this is an isomorphism of varieties; it is enough to show that $\overline{\sigma}$ is a homeomorphism of the underlying topological spaces, and then to show that the infinitesimal maps on the Zarisky tangent spaces are all injective.

We will start by showing that $\overline{\sigma}$ is a homeomorphism. We will need to go about this in a somewhat roundabout way. First, $G\times_P Y$ is a fibre bundle over $G/P$ in a natural way, and consider its product with $\overline{\sigma}$. This gives a monomorphism $\iota: G \times_{P_\beta}S_\beta \to G/P\times X$, and realizes $\overline{\sigma}$ as the second projection $G/P\times X\to X$. In particular, since $G/P_\beta$ is proper, it is now enough to show that the image of $\iota$ is locally closed. To see that this is the case, let $Y:=\overline{S_\beta}$, and consider the product of the projection of the first coordinate $G\times Y\to G/P_\beta$ with the restriction of the action $\sigma:G\times Y\to X$. The image of this map is closed, and $S_\beta$ is easily seen to be open in it, as desired.

Finally, we consider the infinitesimal properties, and it is enough to consider only the distinguished point of $G/P_\beta$ (the others follow by translation.) If we let then $m=(P_\beta, y)$ for $y\in S_\beta$, an element of $T_m(G/P_\beta \times S_\beta)$ is of the form $(a+\mathfrak{p}_\beta, \xi)$ where $a+\mathfrak{p}_\beta\in \mathfrak{g}/\mathfrak{p}_\beta$ and $\xi\in T_yX$ such that $a^\dagger_y+\xi \in T_yS_\beta$. Now, this is in the kernel of the second projection if and only if $a^\dagger_y\in T_yS_\beta$, which implies that $a\in \mathfrak{p}_\beta$, or in other words that $(a+\mathfrak{p}_\beta, \xi)$ is the zero element as desired.
\end{proof}

We want to note that our proof is rather general, using only the properties of the stratification. In particular, it also applies in the Kirwan's projective setting. 

Also, the action of $\lambda$ defines a $\mathbb{C}^*$-action on $V$ which stabilizes $S_\lambda$. In fact, if $p_\lambda:V\to V^\lambda$ is the retration onto the fixed point set, we define
\begin{equation*}
Z_\lambda:=p_\lambda(S_\lambda)
\end{equation*}
We have that $S_\lambda=p^{-1}_\lambda(Z_\lambda)$. We also conclude that
\begin{equation}\label{cohomologiesaffine}
H^*_G=H^*_{P(\lambda)}(S_\lambda)=H^*_{L(\lambda)}(Z_\lambda)
\end{equation}
Just as in the projective case, we can give an intepretation of $Z_\lambda$ as a semistable locus for the action of $\lambda$ on a certain subvariety of $V$.

\subsection{Symplectic reduction}

We may consider $V$ as a K\"ahler manifold (with constant K\"ahler form,) and the action of $G$ satisfies the requirements for the construction of a K\"ahler quotient. A moment map for the action of $K$ is a map $\mu: V\to \mathfrak{k}^*$ which is equivariant with the respect to the coadjoint action of $K$ on $\mathfrak{k}^*$, and which satisfies the condition
\begin{equation*}
\langle d\mu , \beta\rangle = \iota (\eta^\dagger)\omega
\end{equation*}
Here, $\langle \cdot ,\cdot \rangle$ is the canonical contraction on $\mathfrak{k}^*\times \mathfrak{k}$, $\beta^\dagger$ is the vector field induced by the infinitesimal action of $\beta \in \mathfrak{k}$, and $\iota(\cdot)$ is the contraction with the vector field. Under our assumptions, there is a natural choice of moment map determined by the expression
\begin{equation}\label{affinemm}
\langle \mu (x), \beta \rangle =(i\beta x, x)
\end{equation}
for $\beta \in \mathfrak{k}$. (Recall that $\beta$ identifies with an endomorphism of $V$ by the representation.)

The \emph{Marsden-Weinstein (or symplectic) reduction} is the quotient
\begin{equation}
V\sslash_\mu G:=\mu^{-1}(0)/K
\end{equation}
where on the right we mean the actual orbit space. If $K$ acts with finite stabilizers on $\mu^{-1}(0)$, then $G\mu^{-1}(0)$ is actually an open set. On points where $K$ acts freely, the reduction inherits a K\"ahler structure. Our work in the previous chapter shows that this notation is abusive, but not innocently so. To first approximation, we may justify the presence of the group $G$ by remarking that the natural map $\mu^{-1}(0)/K\to G\mu^{-1}(0)/G$ is a homeomorphism. The following definitions are also motivated by our discussion in that chapter.
\begin{definition}
Let $\Delta$ the the intersection of the stabilizers of all points of $X$. A point $x\in X$ is
\begin{enumerate}
\item \emph{$\mu$-semistable} if $\overline{G\cdot x}\cap \mu^{-1}(0)\neq \varnothing$.
\item \emph{$\mu$-polystable} if $G\cdot x \cap \mu^{-1}(0)\neq \varnothing$.
\item \emph{$\mu$-stable} if it is polystable and its stabilizer is precisely $\Delta$.
\end{enumerate}
\end{definition}
The set of $\mu$-semistable points is open, usually rather large, and the symplectic reduction parametrizes its closed orbits.

\subsubsection{The Morse stratification}
Fix a $K$-invariant product on $\mathfrak{k}$, and recall that the real part of the hermitian product defines a Riemannian metric $g$ on $V$ (in fact, just a positive definite quadratic form on $V$ itself.) Define
\begin{equation*}
f(v)=||\mu(v)||^2
\end{equation*}
This defines a smooth function on $V$, and its critical points are determined by the equation $i\mu^* (v) \cdot v=0$; indeed,
\begin{equation*}
(df)_v=2(d\mu (v), \mu (v))=2\langle d\mu (v), \mu^*(v)\rangle=2\omega (i\mu^*(v)\cdot v, v )
\end{equation*} We may also consider paths of steepest descent from any point $v\in V$, namely the solutions $\gamma$ to the ODE problem given by
\begin{equation*}
\frac{d\gamma_v}{dt}(t)=-\nabla f(\gamma_v (t))
\end{equation*}
with initial value $\gamma_v(0)=v$.

Recall that in defining the Morse stratification one needs two assumptions:
\begin{enumerate}
\item The negative gradient flow of $f$ at any point $x\in X$ is contained in some compact neighbourhood of $X$;

\item The critical set $C$ of $f$ is a topological coproduct of a finite number of closed subsets $C_{[\beta]}$, $\beta \in B$, on each of which $f$ takes constant value, and such that $\beta<\beta$ if $f(C_{[\beta]})<f(C_{[\beta']})$ is a strict ordering of $B$.
\end{enumerate}
Compactness, and so projectivity, is a sufficient condition for these two assumptions. But $V$ is an affine space, so that the conditions need to be proven for this case. The proof of the condition on the flows is due to Harada-Wilkin \cite{hw} Lemma 3.3. The assumption on the indices was proven by Hoskins \cite{hoskins} section 3.3.

We will not reproduce the proofs here, but we want to give the description of the indices, from which finiteness will follow. Whereas for a compact symplectic manifold, the image of the moment map is a convex polytope, the imagine of the moment map for an affine space is a polyhedral cone. Indeed, it is the cone generated by the weights of the action of the maximal torus of $K$ on the space $V$, shifted by the vector $\chi^*$ determined by the character $\chi$ by the chosen invariant pairing. The indices of the stratification are then the closest point to the origin of the cone generated by some subset of weights. It follows that there are only finitely many indices $\beta$, and they are all rational in the sense that for some integer $n$, $n\beta$ exponentiates to a one-parameter subgroup.

It follows from finiteness of the indexing set $B$ that if $C$ is a connected component of the critical set of $f$, then $\mu^* (C)$ must lie in a single adjoint orbit of $\mathfrak{k}$ (in fact, map onto it, since the moment map is equivariant with respect to the adjoint action.) To see this, we have just to note that $K$ is compact, so that the adjoint orbits of $\mathfrak{k}$ are all closed. With this in mind, given an element $\beta \in \mathfrak{k}$, we let $C_{[\beta]}$ be the set of critical points $v$ of $f$ with $\mu^*(v)$ conjugate to $\beta$ by $K$, equipped with the subspace topology. If we fix a positive Weyl chamber $W^+$, we conclude that the topological coproduct
\begin{equation*}
C=\coprod_{\beta \in W^+} C_{[\beta]}
\end{equation*}
is actually isomorphic to the critical set of $f$. Also, if for any point $v\in V$ we denote the path of steepest descent by $\gamma_v (t)$, assumption 1 guarantees that $\gamma_v(t)$ converges to a unique point $v_\infty$ which is critical for $f$ (\cite{hw} Lemmas 3.6, 3.7.) It then makes sense to define
\begin{definition}
For any $\beta \in B$, let
\begin{equation*}
S_{[\beta]}:=\left\{x\in X |\phantom{.} x_\infty \in C_{[\beta]} \right\}
\end{equation*}
\end{definition}
It follows that $V$ is the disjoint union of all $S_{[\beta]}$. The following now follows straightforwardly from our work in Chapter 1. Recall that we denote by $S_{[\beta],m}$ the component of the Morse stratum with codimension $m$; this corresponds to the component of the critical set $C_{[\beta]}$ where the Hessian of $f$ has index $m$.\
\begin{theorem}\label{morseequalitiesaffine}
For a suitable, always existing choice of Riemannian metric, the collection $\{S_{[\beta],m}\}$ defined above is an equivariantly perfect smooth stratification of $X$ over the rationals. Therefore, the equivariant Poincar\'e polynomial of $X$ is given by
\begin{equation*}
P_t^K(X)=\sum_{\beta,m} t^{d(\beta, m)}P^K_t(S_{[\beta],m})
\end{equation*}
\end{theorem}

\subsection{The Kempf-Ness and Kirwan-Ness Theorems}
We now establish the relation between the algebraic and symplectic constructions above. Since the various linearizations in principle determine different quotients, it doesn't make sense to compare them to the fixed moment map we defined above for the Marsden-Weinstein reduction. However, the derivative $d\chi$ of the character determines an element in $\mathfrak{k}^*$ which is central, and we define the shifted moment map
\begin{equation*}
\mu^\chi=\mu+d\chi
\end{equation*}
The following is the affine version of the Kempf-Ness Theorem, originally proved by King.
\begin{theorem}[\cite{king} Thm.\ 6.1]
A point $x\in V$ is $\chi$-(semi,poly)stable if and only if it is $\mu^\chi$-(semi,poly)stable. Consequently, $V\sslash_\chi G$ and $V\sslash_{\mu^\chi}G$ are homeomorphic.
\end{theorem}

Note that we could have alternatively have seen the characters as determining symplectic reductions at different level sets of the same fixed moment map.

This coincidence extends from the quotients to the stratifications. Suppose $\beta$ is an index for a Morse stratum. There is an integer $n$ such that $n\beta$ is an integral point of $\mathfrak{k}$, and so defines a one-parameter subgroup $\lambda_\beta$. This turns out to index a Hesselink stratum. In fact, Hoskins established the following analogue of the Kirwan-Ness Theorem.

\begin{theorem}[\cite{hoskins} Thm.\ 4.12]
The Morse stratum $S_{[\beta]}$ and the Hesselink stratum $S_{[\lambda_\beta]}$ coincide. 
\end{theorem}
We now bootstrap Theorem \ref{morseequalitiesaffine} with (\ref{cohomologiesaffine}) to establish the formula
\begin{equation}
P_t^G(V^{ss})=P_t(BG)-\sum_{\beta \neq 0, m}t^{d(\beta,m)}P_t^{L(\beta)}(Z_{\beta,m})
\end{equation}

\section{GIT for generalized quivers}\label{gitquot}

\subsection{Linearizations and moment maps for generalized quivers}
Let $\tilde{Q}=(R,\Rep (\tilde{Q}))$ be a generalized $G$-quiver, and fix a character $\chi:G\to \mathbb{G}_m$ as well as maximal compacts $K_R\subset K$ of $R$ and $G$, respectively. Given an element $\beta \in \mathfrak{k}_R$, we can define two Levi subgroups
\begin{align*}
L_R(\beta) &=\{ g\in R |  \exp(it\beta) g \exp (-it\beta)=g \} \\
L (\beta) &=\{ g\in G | \exp (it\beta) g \exp (-it\beta)=g \}
\end{align*}
of $R$ and $G$ respectively. We trivially have $L_R(\beta)\subset L_G (\beta)$, and so $L_R(\beta)$ has a restricted adjoint action on $\mathfrak{l}_G(\beta)$. This action coincides with the action on $\mathfrak{g}$, where $\mathfrak{l}(\beta)$ is an invariant subspace. Given a decomposition of $\Rep (\tilde{Q})=\bigoplus Z_\alpha$ as an $\mathrm{Ad}\phantom{.}R$-module, it then makes sense to consider the intersection $Z_\alpha(\beta)=Z_\alpha\cap \mathfrak{l}(\beta)$ of \emph{modules}, since $Z_\alpha$ is isomorphic to a unique irreducible piece of the module $\mathfrak{g}$, and define $\Rep (\tilde{Q}_\beta):=\bigoplus Z_\alpha(\beta)$. We then have:

\begin{lemma}\label{levigenquiver}
As defined above, $\tilde{Q}_\beta = (L_R(\beta), \Rep (\tilde{Q}_\beta))$ is a generalized $L(\beta)$-quiver. If $\tilde{Q}$ is a quiver of type $Z$ with $R=Z_G(H)$, then $\tilde{Q}_\beta$ is of type Z and $L_R=Z_L(H)$.
\end{lemma}

Note that this new generalized quiver is independent of the particular $\beta$ we pick to realize $L=L(\beta)$, and we could just have started with an arbitrary Levi $L\subset G$ such that $L_R=L\cap R$ is a Levi of $R$; we'll often speak of $\tilde{Q}_L$ when we don't want to emphasize $\beta$. We can give an interpretation of this result by fixing a faithful representation $K\to U(V)$. We've seen that $\beta$ determines a grading of $V$, and that elements of the Levi subgroups above are precisely those that stabilized the splitting. On the other hand, under the identification of $\mathfrak{g}$ as endomorphisms of $V$, the elements of $Z_\alpha (\beta)$ are precisely those of $Z_\alpha$ which also split as graded endomorphisms of $V$. We can then make sense of subrepresentations of the original representation, so that the representations of $\tilde{Q}_\beta$ are precisely the splittings of representations of $\tilde{Q}$ according to the action of $\beta$. It is useful to keep this interpretation in mind as we discuss stability, and when discussing classical quivers we'll be able to see this splitting very explicitly (also in that setting the abelian group's role in the story will become apparent.)
 
Many of our results will relate stability properties of representations of $\tilde{Q}$ with those of $\tilde{Q}_L$, so we will need to define a suitable linearization for $\tilde{Q}_L$ starting from the character $\chi$. Now, $\chi$ is of course a character of $L$ itself, but it is not suitable for the following simple reason: if $L=L(\beta)\neq G$, whereas $\exp (\beta)$ is in the kernel of the representation of $L_R$ on $\mathfrak{l}$, it is not on the kernel of the representation of $R$ on $\mathfrak{g}$. It is therefore perfectly possible that $\chi (\exp(\beta))\neq 1$, which we'll see below makes the stability condition for $\tilde{Q}_L$ empty. In this setting, we need to correct the choice of character for $\chi_L$ on $\Rep (\tilde{Q}_L)$ by projecting out the new elements in the kernel of the representation. For this, we choose an $L_R$ invariant inner product $(\cdot,\cdot)$ on $\mathfrak{l}_R$; since $\mathfrak{z}(\mathfrak{l})\subset \mathfrak{z}(\mathfrak{l}_R)$ is invariant, we use the inner product to choose a complementary space, and denote $p_{\mathfrak{z}(L)}$ the projection onto that complemente; we then define $\chi_L$ as the character determined by $(\chi_L)_*=\chi_*\circ p_{\mathfrak{z}(L)}$. For the symplectic point of view, if $\mu_s$ determines the standard moment map given by the hermitian metric, we've see that the moment map corresponding to $\chi$ is $\mu=\mu_ s -\chi_*$; we conclude then that the moment map adapted to $\chi_L$ is $\mu_L=\mu-\chi_*\circ p_{\mathfrak{z}(L)}$.

\subsection{Stability properties} 
We'll begin by characterizing the convergence for one-parameter subgroups. Let $\tilde{Q}=(R, \Rep (\tilde{Q}))$ be a generalized $G$-quiver, and fix a character $\chi$ of $G$.
\begin{lemma}\label{onepargen}
Let $\varphi \in \Rep(\tilde{Q})$ be a representation of $\tilde{Q}$, and let $\lambda$ be a one-parameter subgroup. 
\begin{enumerate}
\item The limit $\lim \lambda(t)\cdot \varphi$ exists if and only if $\varphi_\alpha \in \mathfrak{p}_G(\lambda)$ for all $\alpha$.
\item If it exists, $\varphi_0:=\lim \lambda (t)\cdot \varphi \in \Rep (\tilde{Q}_\lambda)\subset \Rep (\tilde{Q})$.
\item If $\varphi_0$ is semistable  as a representation of $\tilde{Q}$, then $\varphi$ is semistable.
\item Suppose $\langle \chi , \lambda' \rangle = 0$ for all $\lambda'$ such that $P_R(\lambda')=P_R(\lambda)$, and that $\varphi_0$ exists. Then, $\varphi$ is semistable if and only if $\varphi_0$ is a semistable representation of $\tilde{Q}_\lambda$.
\item Under the conditions of the last point, if $\varphi$ is semistable, then $\varphi_0$ is a semistable representation of $\tilde{Q}$.
\end{enumerate}
\end{lemma}

\begin{proof}
(1) is obvious from the definition of $\mathfrak{p} (\lambda)$. To prove (2), it is enough to prove that $\varphi_{0,\alpha} \in \mathfrak{l}_G(\lambda)$. We have
\begin{equation*}
\Ad (\lambda(t))\varphi_0= \Ad (\lambda(t))\lim_{u\to 0} \Ad (\lambda(u))\varphi=\lim_{u\to 0} \Ad (\lambda(ut))\varphi=\varphi_0
\end{equation*}
Point (3) follows from the fact that the set of unstable points is closed.

Now, point (4) from the fact that all such $\lambda'$ generate the center of $L(\lambda)$, and so the condition in the theorem ensures that the character $\chi_L$ on $\Rep (\tilde{Q}_\lambda)$ is precisely $\chi$. Point (5) follows immediately from this, since the coincidence of the characters guarantees that semistable points of $\Rep (\tilde{Q}_\lambda)$ are sent to semistable representations of $\tilde{Q}$.
\end{proof}

One must be careful in interpreting this lemma. Let $\varphi$ be a representations, and supposed $\lim \lambda(t)\cdot \varphi$ exists. Given point (1) above, and since a parabolic is determined by its strictly dominant elements,\footnote{Recall here that dominant elements are dual of dominant characters in $\mathfrak{z}\oplus \mathfrak{c}$, cf. section \ref{parabolic}.} one might be tempted to conclude that $\lim \lambda'(t)\cdot \varphi$ exists for any dominant $\lambda'$ of $P_R(\lambda)$. However, this does not quite follow from (1), because we need $\lambda'$ do be dominant for $P_G(\lambda)$, \emph{not} $P_R(\lambda)$! There is in fact a difference in the components along the center of $P_R(\lambda)$: a dominant for this latter group has an arbitrary component along the center, whereas dominants for $P_G(\lambda)$ don't (they are only arbitrary along the smaller center of $P_G(\lambda)$ itself.) We will use the term \emph{$G,R$-dominant} to refer to dominants of both $P_R(\lambda)$ and $P_G(\lambda)$; or equivalently, for dominants of $P_G(\lambda)$ which happen to belong to $P_R(\lambda)$.

What is clear is that the stability condition is not really a matter of the one-parameter subgroups, but rather on the parabolics themselves. This is made clear in the next proposition, which resembles stability conditions in gauge theory.

\begin{proposition}\label{gaugestability}
Let $\varphi \in \Rep (\tilde{Q})$ be a representation, and $\mathscr{P}(\varphi)$ be the set of parabolics $P$ of $G$ such that $\varphi \in \mathfrak{p}$, and $P=P(\lambda)$ for some OPS $\lambda$ of $R$. Then, $\varphi$ is
\begin{enumerate}
\item semistable if and only if for every $G,R$-dominant weight $\beta$ of $P\in \mathscr{P}(\varphi)$ we have $\chi_*(\beta)\geq 0$.
\item stable if and only if for every $G,R$-dominant weight $\beta$ of $P\in \mathscr{P}(\varphi)$ we have $\chi_*(\beta)> 0$.
\end{enumerate}
\end{proposition}
\begin{proof}
The only thing we need to prove is that it is enough to check that the Hilbert-Mumford pairing can be computed with the derivative of the character. If $\beta$ is integral, then this is the following computation:
\begin{equation*}
\chi (\lambda_\beta (e^s))=\chi (\exp (s\beta))=e^{s\chi_*(\beta)}=t^{\chi_*(\beta)}
\end{equation*}
Otherwise, there is always a positive integer $n$ such that $n\beta$ is an integral point, and we have $\langle \chi, \lambda_{n\beta}\rangle = \chi_*(n\beta) =n\chi_* (\beta)$, so that the sign doesn't really change.
\end{proof}

We could proceed along these lines, but since the final result is not very enlightening except for the case of type Z quivers, we'll abstain from further discussion at this point. Below we study the type Z case in detail, and the significance (as well as naturality) of this dependence on parabolics only will become clear.

\subsection{Jordan-H\"older objects}
\begin{definition}
A pair of parabolic subgroups $(P_R\subset R,P\subset G)$, is \emph{admissible} if $P\cap R=P_R$ and $\langle \chi , \lambda' \rangle = 0$ for all OPS $\lambda'$ \emph{of the group $R$} such that $P(\lambda')=P$.
\end{definition}
We'll need the existence of admissible parabolics below.
\begin{lemma}
If $\varphi$ is strictly semistable, then $\varphi \in P$ for some admissible pair $(P_R,P)$. Furthermore, there is a minimal such admissible $P$.
\end{lemma}
\begin{proof}
If $\varphi$ is strictly semistable, then there is a one-parameter subgroup $\lambda_\beta$ with $\langle \chi, \lambda_\beta \rangle=0$ for which the limit exists. The restriction that $\lambda_\beta$ be strictly dominant for both $P(\lambda_\beta)$ and $P_R(\lambda_\beta)$ is precisely that there is a decomposition
\begin{equation*}
\beta=z_\beta+\sum_j \beta_j\alpha^G_j + \sum_i \beta_i \alpha_i
\end{equation*}
where $z_\beta \in \mathfrak{z}(G)$, $z_\beta+\sum_j \alpha_j^G \in \mathfrak{z}(R)$ with $\alpha_j^G$ corresponding to positive combinations of simple weights corresponding to $P_G(\lambda)$, and the $\alpha_i$ are the simple weights corresponding to $P_R(\lambda)$. From the fact that $\beta$ is strictly dominant for both $P_R(\lambda)$ and $P_G(\lambda)$ it follows that $\beta_j<0$ and $\beta_i<0$. We also assumed that $\chi$ is trivial on the center of $G$ to ensure the existence of semistable points, so that $\chi_* (z_\beta)=0$; it follows that if $\chi_*(\beta)=0$, then $\chi_*(\alpha_j^G)=\chi_*(\alpha_i)=0$. But every dominant of $P_R(\beta)$ can be expressed in terms of the same $\alpha_j^G$ and $\alpha_i$, so $P_R(\lambda)$ is admissible. 

To prove that there is a minimal admissible, it is enough to remark the following: if $P_1$ and $P_2$ are admissible, and defined by sets of simple roots $A_1$ and $A_2$, respectively, then $A_1\cup A_2$ defines an admissible parabolic smaller than both. This is enough since this reduces the semisimple rank, which is finite to start with.
\end{proof}

We will also need the following result:
\begin{lemma}\label{indadmissible}
Suppose $(P_R,P)$ is admissible for $\varphi$, and let $P_R=L_R U_R$ and $P=LU$ be the Levi decompositions with $L_R=L\cap R$. If $(P'_R\subset L_R, P'\subset L)$ is admissible for $\varphi_0=\lim \lambda(t)\cdot \varphi$, then $(P'_RU_R, P'U)$ is admissible for $\varphi$.
\end{lemma}
\begin{proof}
This follows from the fact that for every dominant $\beta$ of $P_1U$ that is some integer $n$ such that $n\beta=\beta_1 + \beta'$ where $\beta_1$ is a dominant of $P_1$ and $\beta'$ is a dominant of $P$ (cf. \cite{ramanathan} 3.5.9.)
\end{proof}

Recall that two points of $V$ are \emph{S-equivalent} if their orbit closures intersect, or, alternatively, both closures share a (necessarily unique) closed orbit. Jordan-H\"older objects select a representative in the closed orbit of each S-equivalence class, and can now be constructed along standard lines by an inductive process. 

The next result determines the existence of Jordan-H\"older objects for generalized quivers.

\begin{proposition}\label{jhgen}
Let $\varphi\in \Rep (\tilde{Q})$ be a semistable representation. Then, there is a parabolic subgroups $P_R\subset R$ and $P\subset G$, $P_R=P\cap R$ with Lie algebra $\mathfrak{p}$ such that $\varphi_\alpha\in \mathfrak{p}$, and if $p:P\to L$ is the projection onto a Levi subgroup, $\varphi_{JH}:=p_* (\varphi)$ is a stable representation of $\tilde{Q}_L$. Furthermore, under the inclusion as a representation of $\tilde{Q}$, $\varphi_{JH}$ is polystable and S-equivalent to $\varphi$.
\end{proposition}
We should remark here that generally speaking, \emph{closed} orbits on the boundary of $R\cdot \varphi$ can always be reached by some one-parameter subgroup. However, the statement in the theorem is stronger insofar as it determines another quiver setting for which the Jordan-H\"older object is stable.
\begin{proof}
If $\varphi$ is stable, nothing needs to be proven. Otherwise, take a minimal admissible parabolic $(P^{min}_R,P^{min})$ for $\varphi$. Let $L_R$ and $L$ be Levis of $P_R$ and $P$, respectively, with $L_R=L\cap R$, and let $p_{min}:P^{min}\to L^{min}$ be the projection. We claim that the $L^{min}_R$ representation $\varphi_{JH}:=p_{min}(\varphi)$ is stable; it is certainly semistable by Lemma \ref{onepargen}. On the other hand, if we assume it is not stable, it admits a pair of parabolic $(P_R \subset L_R^{min}, P\subset L^{min})$. But by Lemma \ref{indadmissible} we have seen that then we can from $(P_R,P)$ construct an admissible pair $(P'_R\subset P_R^{min}, P'\subsetneq P^{min})$, which is a contradiction. As a $G$-representations, $\varphi_{JH}$ is certainly S-equivalent to $\varphi$, since this projection is the limit of the flow by $\lambda_\beta$ for some dominant $\beta$ of $P_R^{min}$, and so the closures of the two orbits intersect. Finally, we must prove that again as a $G$-representation it is polystable, i.e., that the orbit $R\cdot \varphi_0$ is closed. But in fact, this orbit is the image of $L_R(\lambda)\cdot \varphi_0$ under the action of $R/P_R(\varphi)$; the latter is proper and the former is a closed subset of $\Rep (\tilde{Q}_{min})$, so the image is also closed since it is the action of a proper group.
\end{proof}

\begin{remark}
Note that we can reach a minimal admissible parabolic by successively considering maximal admissible parabolics, and so arrive at an inductive process which more closely resembles the usual construction of Jordan-H\"older objects. To make this precise, assuming that $\varphi$ is strictly semistable, choose a maximal admissible parabolic $P_1$; from Lemma \ref{onepargen}, we conclude that $\varphi_1:=p_1(\varphi)$ is semistable. If it is stable, we are done; otherwise, choose a parabolic $P_2$ in $L_1$ that is maximally admissible for $\varphi_1$ and repeat. Since the semisimple rank keeps decreasing and also generalized quivers determind by tori are automatically stable, the process must stop at a finite number of steps. That this is the same as above follows again by the construction above for each $P_i\subset L_{i-1}$ of an parabolic $P'_i\subset G$ that is admissible for $\varphi$. We conclude that the process stops precisely when $P'_i$ is a minimal admissible. 
\end{remark}

\begin{corollary}
Two representations $\varphi$ and $\varphi'$ are S-equivalent if and only if there is an $r\in R$ such that $\varphi_{JH} = r\cdot \varphi'_{JH}$.
\end{corollary}

\subsection{The local structure of the quotient}

We will now investigate the local structure of the quotient, starting with the deformation theory of generalized quivers.

\begin{lemma}
Let $\varphi$ be a polystable representation. The deformation space $N_\varphi$ of $\varphi$ is a representation space of a $G$ generalized quiver $\tilde{Q}_\varphi$ with symmetry group $R_\varphi:=\Stab(\varphi)$
\end{lemma}
\begin{proof}
Since our variety is an affine space, this reduces to the following sequence of vector spaces:
\begin{equation*}
0\longrightarrow \mathrm{ad}(\mathfrak{r})\varphi \longrightarrow \Rep (\tilde{Q}) \longrightarrow N_\varphi \longrightarrow 0
\end{equation*}

Picking a hermitian metric, we can now find a splitting of $\Rep (\tilde{Q})$ which is also a splitting as an $\ad (\mathfrak{r})$-module, which allows us to identify $N_\varphi$ as a subspace of $\Rep (\tilde{Q})$. On the other hand, the action of $R_\varphi:=\Stab (\varphi)$ respects  this splitting, so that $N_\varphi$ is a sum of $\Ad (\Stab (\varphi))$-submodules of $\mathfrak{g}$.
\end{proof}

A immediate application follows by Luna's results \cite{luna1} III.1.

\begin{theorem}
There is an \'etale map from a neighbourhood of the origin in $\Rep (\tilde{Q}_\varphi)\sslash R_\varphi$ to a neighbourhood of $\varphi$ in the quotient $\Rep (\tilde{Q})\sslash R$.
\end{theorem}
Since we're working with complex varieties, recall that this result in particular implies that there is a biholomorphism between neighbourhoods of the points in question in the classical topology.

\begin{remark}
Given our characterization of Jordan-H\"older objects above, one might be tempted to try to characterize the Luna strata in terms of certain Levi subgroups (especially since something of the sort can be accomplished for the Hesselink strata, as we'll see below.) However, a more careful analysis easily shows that this is not something we can expect to be possible, as the Luna stratification depends on stabilizers of representations, and a characterization of stabilizers will in general involve reductive subgroups that are smaller than Levis.
\end{remark}

\subsection{The Hesselink stratification}\label{hesselink}
Characterizing the instability type of generalized quivers involves finding a suitable characterization of $P(\varphi)$ which will involve $\Rep (\tilde{Q}_L)$ for some Levi $L$ of $P(\varphi)$. We'll denote by $P_G(\varphi)$ the parabolic of $G$ determined by a most destabilizing OPS. Note that the character for $\tilde{Q}_\lambda$ induced by $\chi$ is in this instance very simple. In fact, if we identify $\chi_*$ with a vector in $\mathfrak{k}$ through the invariant inner product, we may think of out definition of $\chi_\lambda$ as simply projecting out the component along the center of $L(\lambda)$; but $L=L(\beta)$  for $\lambda=\lambda_\beta$ precisely means that it is the closest element to $\chi$ (up to scalar factors,) so that the component along the center is the component along $\beta$. We conclude then that in this instance we can write
\begin{equation*}
(\chi_L)_ *=\chi_* -m_\chi(\lambda_\beta)\beta^*
\end{equation*}
where $\beta^*$ is the dual of $\beta$ by the invariant inner form. (To get an integral point we should then multiply by $||\beta||$.)

With this preliminary, we can now prove the following.
\begin{theorem}\label{instparalg}
Let $\varphi$ be a representation, and $(P_G(\lambda), P_R(\lambda))$ be a pair of parabolics with $\varphi \in \mathfrak{p}$, and such that
\begin{enumerate}
\item If $p:\mathfrak{p}_G\to \mathfrak{l}_G$ is the projection onto a Levi, $p(\varphi)$ is a semistable representation of $\tilde{Q}_L$; and
\item For every $G,R$-dominant element $\beta \in \mathfrak{p}$, we have $\chi_*(\beta) >0$. 
\end{enumerate}
Then, $\varphi$ is unstable, $P_G(\lambda)=P_G(\varphi)$, and $P_R(\lambda)=P(\varphi)$. Conversely, if $\varphi$ is unstable, $P(\varphi)$ satisfies the properties above.
\end{theorem}
\begin{proof}
Let $\varphi$ be an unstable representation, $\lambda_\beta$ a most destabilizing OPS (so that $P(\varphi)=P(\beta)$,) and suppose that $p(\varphi)$ is unstable. In particular, there is an OPS $\lambda_{\beta'}$ of $L$ which destabilizes $p(\varphi)$ and we have
\begin{equation*}
\chi_*(\beta +\beta')=\chi_*(\beta)+ (\chi_L)_*(\beta')+m_\chi(\beta)(\beta, \beta')
\end{equation*}
From this, we can deduce that if $\beta$ is taken in the same Cartan subalgebra as $\beta'$ (which we can always do,) then $\langle \chi, \lambda_\beta \lambda_{\beta'} \rangle \geq \langle \chi, \lambda_\beta \rangle$, which is a contradiction. Therefore, $p(\varphi)$ is semistable. On the other hand, if $\beta=\sum \beta_i \alpha_i$ is the decomposition of $\beta$ into the simple weights of $P_G(\lambda)$, we have that $\beta_i<0$ for all $i$ because $\beta$ is strictly dominant. Property (2) will follow if we prove that $\lambda_\beta$ being most destabilizing implies that $\chi_*(\alpha_i)<0$ (recall that the simple weights are \emph{anti}dominant, not dominant!) But indeed, suppose $\chi_*(\alpha_j)\geq 0$ for some $j$, and let $\beta'=\sum_{i\neq j} \beta_i \alpha_i$. We certainly have $m_\chi(\beta')\geq m_\chi (\beta)$, so that we will obtain a contradiction if we prove that $\lim \lambda_{\beta'}(t)\cdot \varphi$ exists. To prove this it is enough to remark that $P_G(\beta)\subset P_G(\beta')$.

Conversely, suppose $P$ satisfies (1) and (2). The first property immediately implies that $\varphi$ is unstable, since any strictly $G,R$-dominant character yields a destabilizing one-parameter subgroup. Suppose $P_R(\lambda)\neq P(\varphi)$, and let $\lambda_\beta$ be a most destabilizing OPS for $\varphi$ in a maximal torus contained in $P_R(\lambda)$. We then have $\langle \chi_L, \lambda_\beta \rangle>0$, so that $p(\varphi)$ is not semistable, a contradiction.
\end{proof}

We will in particular need the following easy consequence.
\begin{corollary}
Let $\varphi$ and $\varphi'$ be two unstable representations of $\tilde{Q}$. Then, $P(\varphi)=P(\varphi')$ if and only if $\varphi'\in p^{-1}(\Rep (\tilde{Q}_L)^{ss})$ (and vice versa.) 
\end{corollary}

Another interesting corollary of the proof is the following:
\begin{corollary}
$(P_G(\varphi), P(\varphi))$ is the maximal pair of parabolics with property (1) in the theorem.
\end{corollary}
\begin{proof}
In the course of the proof we proved the following: if a parabolic satisfies (2) and is not the maximal destabilizing, then it does not satisfy (1). In other words, if $P$ is a parabolic determined by a set $A$ of simple roots which satisfies (1), then necessarily there is $\alpha \in A$ such that $\chi_*(\alpha)\geq 0$, unless $P=P(\varphi)$. Let $P'$ be the parabolic determined by $A-\{ \alpha \}$; by definition, $P\subset P'$, and $P'$ satisfies (1). To see this, note that $(\chi_{L'})_*=(\chi_L)_* - m_\chi (\alpha)\alpha^*$ so that they differ only along the dominants colinear with $\alpha$, and for those we have $(\chi_{L'})_*(\beta)\geq (\chi_L)_ *(\beta)$ because both $m_\chi(\alpha)$ and $-(\alpha, \beta)$ are positive.
\end{proof}

Our aim is an explicit description of the Hesselink stratification. For this we use Lemma \ref{bladeaffine} and Theorem \ref{instparalg} and its corollaries. In particular, we have the following result.
\begin{theorem}\label{instops}
Let $\beta \in \mathfrak{r}$, and $S_{[\beta]}$ the corresponding Hesselink stratum in the space of representations. Then, $H^*_G(S_{[\beta]})=H^*_{L(\beta)}(\Rep(\tilde{Q}_{L(\beta)})^{ss})$
\end{theorem}
\begin{proof}
From Lemma \ref{bladeaffine} we have $H^*_G(S_{[\lambda]})=H^*_{P(\lambda)}(S_\lambda)$. On the other hand, on $S_\lambda$ we have the action of $\mathbb{G}_m$ through the OPS $\lambda$, so that by a theorem of Bialinicky-Birula \cite{bb}, there is a retraction onto the fixed point set. But it follows from \ref{instparalg} that this fixed point set is precisely $\Rep^{ss}(\tilde{Q}_{L(\beta)})$, so that we have an isomorphism $H^*_{P(\lambda)}(S_\lambda)=H^*_{L(\beta)}(\Rep(\tilde{Q}_{L(\beta)})^{ss})$.
\end{proof}

In other words, the cohomology of each stratum can be computed in terms of the quotient of representations for a Levi of $G$, which has lower semisimple rank. We shall see below that in conjunction with Morse theory, this in fact yields a suitable inductive formula for equivariant cohomology.

\section{Morse theory for generalized quivers}\label{morse}

Let $G$ be a complex reductive Lie group, $K_G$ a maximal compact of $G$, and $\tilde{Q}=(R, \Rep (\tilde{Q}))$ be a generalized $G$-quiver; pick a maximal compact $K_R$ of $R$ such that $K_R\subset K_G$. By definition, we have a decomposition $\Rep=\bigoplus \mathfrak{g}_\alpha$ as an $R$-module, where $\mathfrak{g}_\alpha \subset \mathfrak{g}$ is a complex subspace. Now, given any complex reductive group $G$, there is a choice of Hermitian metric on its Lie algebra $\mathfrak{g}$ such that $K$ acts unitarily, and this metric restricts to each $\mathfrak{g}_\alpha$. This implies that $\Rep (\tilde{Q})$ is a Hermitian space. Further, since $K_R\subset K_G$, the group $K_R$ acts unitarily on each $\mathfrak{g}_\alpha$, and consequently also on $\Rep (\tilde{Q})$. For the case of a classical group, all of these are induced by a choice of a Hermitian metric for the standard representation.

We can now, therefore, apply the methods we just reviewed above to conclude that there is a naturally defined moment map $\Rep (\tilde{Q}) \to \mathfrak{k}^*_R$ for the action of $K_R$. Again, denote by $f$ the square of the moment map, i.e., $f(\varphi)=||\mu(\varphi)||^2$. Let $\beta \in i\mathfrak{k}_R$, and recall from section \ref{preliminaries} the Levi subgroups
\begin{align*}
L_R(\beta) &=\{ g\in R |  \exp(i\beta t) g \exp (-i\beta t)=g \} \\
L_G (\beta) &=\{ g\in G | \exp (i\beta t) g \exp (-i\beta t)=g \}
\end{align*}
Using these, we can characterize the critical points of the square of the moment map. Recall for the next proposition that we have systematically defined moment maps for each generalized quiver determined by a Levi in section \ref{parabolic}. In the case when $\beta$ determined critical components, we can actually simplify that moment map, like we did for the Hesselink stata. In fact, it is easy to see that the moment map for $\beta$ indexing a critical stratum is just
\begin{equation*}
\mu_\beta=\mu - \beta^*
\end{equation*}
where $\beta^*$ is the dual of $\beta$ through the fixed invariant inner product.

\begin{proposition}\label{criticalgeneralized}
Let $\varphi \in \Rep (\tilde{Q})$, and $\beta = \mu (\varphi)$. Then, $\varphi$ is a critical point of $f$ if and only if $\varphi$ defines a zero of the moment map as a generalized $L_G(\beta)$-quiver representation of $\tilde{Q}_\beta$.
\end{proposition}
\begin{proof}
We've seen above that the critical points $\varphi=(\varphi_\alpha)$ are determined by
\begin{equation*}
i\beta \cdot \varphi_\alpha= \mathrm{ad}(i\beta)\varphi_\alpha =0
\end{equation*}
for each $\alpha$. Since $i\beta \in \mathfrak{k}_R \subset \mathfrak{k}_G$, we immediately conclude that if $\varphi$ is a critical point of $f$, $\varphi_\alpha \in \mathfrak{l}_G (i\beta)$ for all $\alpha$. That this is a zero of the moment map then follows by definition. Conversely, if $\varphi$ defines a zero of the moment map as a $\tilde{Q}_\beta$-representation, then $\mu^*(\varphi)=\beta$, and since $\varphi \in \mathfrak{l}(\beta)$, we necessarily have $i\mu^*(\varphi) \cdot \varphi=0$.
\end{proof}

We will abstain from further descriptions of the strata in Morse theoretic terms, since we know it coincides with the algebraic one above, and it is certainly more natural in that language. We make only one further comment: whereas the Hesselink and Morse strat coincide, this is not true at all of the corresponding `critical strata.' We can see this very clearly from this proposition: the `critical Hesselink stratum' corresponding to a $\beta$ (i.e., $S_\beta$) is the set of semistable representatiosn of the appropriate $L_G(\beta)$-quiver; the Morse critical stratum, on the other hand, is only the set of polystable such representations.

We want to apply the general results to the case of generalized quivers. For that, we first need to compute codimensions for the Morse strata, and it turns out that in this case the codimension is constant for each entire stratum. To see this, recall first Lemma \ref{bladeaffine}, which implies that
\begin{equation*}
\dim S_{[\beta]}=\dim (R/P_R(\beta))+\dim S_\beta
\end{equation*}
Now, we have identified $S_\beta$ as $p^{-1}(\Rep (\tilde{Q}_\beta)^{ss})$, where $p: \mathfrak{p}(\beta) \to \mathfrak{l}(\beta)$ is the projection onto the Levi. But this is an open set in $\Rep (\tilde{Q})\cap \mathfrak{p}(\beta)$, so that dimension is constant. Further, we identify $\mathfrak{k}/\mathfrak{p}(\beta)\simeq \mathfrak{u}(-\beta)$ and $R/P(\beta)\simeq U(-\beta)$, where $U(-\beta)$ and $\mathfrak{u}(-\beta)$ are respectively the unipotent radical of $P(\beta)$ and its Lie algebra (i.e., the nilpotent radical of $\mathfrak{p}(-\beta)$.) We then find
\begin{align*}
\mathrm{codim}\phantom{.} S_{[\beta]}&=\dim \Rep (\tilde{Q})-\dim (\Rep (\tilde{Q})\cap \mathfrak{p}(\beta))+\dim (R/P_R(\beta))\\
&=\dim ( \Rep (\tilde{Q})\cap \mathfrak{u}(-\beta) - \dim \mathfrak{u}(-\beta)
\end{align*}

We have therefore shown the following.
\begin{theorem}\label{indcohom}
Let $H$ be the set of $\beta \in W^+$ indexing the Morse stratification. Then,
\begin{equation*}
P_t^G(\Rep(\tilde{Q})^{ss})=P_t(BG)-\sum_{\beta \in H}t^{2d(\beta)}P_t^{L(\beta)}(\Rep (\tilde{Q}_\beta)^{ss})
\end{equation*}
where $d(\beta)=\dim ( \Rep (\tilde{Q})\cap \mathfrak{u}(-\beta) )-\dim \mathfrak{u}(-\beta)$.
\end{theorem}

We finish by noting that following Harada-Wilkin, one can now use the flow of the Morse map to obtain local coordinates at any point. We will not, however, do this at present.

\section{Classical Quivers}\label{classical}

This section is a sort of extended example: we will mostly state without proof well-known results about classical quivers, and point out how they correspond to results above. It serves, however, not just to illustrate the previous material, but also as a guide to the study of symmetric quivers that follows.

\subsection{Classical quivers as generalized quivers}

We start by exposing the tight relations between generalized quivers, which are Lie theoretic entities, and the classical theory of quivers, which come from `graphical interpretations.' In fact, plain quivers, with which we start, were the very motivation for generalized quivers.

\begin{definition} 
Let $\mathrm{\textbf{Vec}}$ be the category of finite dimensional complex vector spaces.
\begin{enumerate}
\item A \emph{quiver} $Q$ is a finite directed graph, with set of vertices $I$, and set of arrows $A$. We let $t:A\to I$ and $h:A\to I$ be the tail and head functions, respectively.

\item A representation $V$ of $Q$ is a realization of the diagram $Q$ in $\mathrm{\textbf{Vec}}$; equivalently, a representation is an assignment of a vector space $V_i$ for each vertex $i\in I$, and a linear map $\varphi_\alpha: V_{t(\alpha)}\to V_{h(\alpha)}$ for every arrow $\alpha$.
\end{enumerate}
\end{definition}

Given a representation of a quiver, let $n_i=\dim V_i$; we call the vector $\mathbf{n}=(n_i)\in \mathbb{N}_0^I$ the \emph{dimension vector} of the representation. It is clear that two representations of $Q$ can only be isomorphic if they have the same dimension vector. Therefore, we always consider this vector as given and fixed. Then, a representation of $Q$ with a prescribed dimension vector is precisely a choice of an element in
\begin{equation*}
\Rep (Q, \mathbf{n})=\bigoplus_{\alpha \in A} \Hom (V_{\ta}, V_{\ha})
\end{equation*}

On this space, we have a clear action of the product group $G(\mathbf{n})=\prod \GL (V_i)$ acting by the apropriate conjugation, namely, an element $g=(g_i)\in G(\mathbf{n})$ acts as $g\cdot \varphi=(g_{\ha}\varphi_\alpha g^{-1}_{\ta})$. The classical theory of quiver representations is precisely the construction of a suitable quotient for this action.

Consider now the direct sum $V=\bigoplus V_i$; this is called the \emph{total space} of the representation.  It is clear that $\Hom (V_i, V_j)$ can be considered as a subspace of $\End (V)$, by extending every element by zero, so that in fact given an arrow $\varphi_\alpha$ in the representation, $\varphi_\alpha\in \End (V)$. In the same way, any automorphism $g_i \in \GL (V_i)$ can be seen as an element in $G(\mathbf{n})$, extending it by the identity; in fact, the whole group fits in, that is, $G(\mathbf{n})\subset \GL (V)$. If we let $H=\left\{ \prod \lambda_i \mathrm{id}_i | \lambda_i \in \mathbb{C}^* \right\}$, we can characterize $R=G(\mathbf{n})$ precisely as the centralizer of $H$ in $\GL(V)$. Further, under the adjoint action of $R$, the Lie algebra $\mathfrak{g}$ of $G$ decomposes precisely as $\mathfrak{g}=\bigoplus_{i,j} \Hom (V_i, V_j)$. It is now clear that $(H, G(\mathbf{n}), \Rep (Q, \mathbf{n}))$ determines a generalized $\GL (V)$-quiver.

\begin{theorem}
There is a bijective correspondence between generalized $\GL (V)$-quivers $\tilde{Q}$ of type Z and classical quivers $Q$ together with dimension vectors $\mathbf{n}$ in such a way that $\Rep (\tilde{Q})= \Rep (Q, \mathbf{n})$.
\end{theorem}

\subsection{Stability of classical quivers}

As we've mentioned, the results of this section are all established and well known, and we refer to \cite{king} and \cite{reineke} for more details. We will deduce the results, however, from our general set up for generalized quivers, hoping to exemplify in a familiar setting the meaning of the results we just obtained. The new ingredient here is the flag interpretation of parabolic subgroups, cf. section \ref{parabolic}.

We will start by showing how the slope-stability criterion arises naturally from our results. We need to start by fixing a total space $V=\bigoplus V_i$, and we denote $n_i:=\dim V_i$ (the vector $\mathbf{n}=(n_i)$ is the \emph{dimension vector}.) The group of symmetries determined by such a total space is $\GL (\textbf{i}):=\prod \GL (V_i)$; a character of this group is of the form
\begin{equation*}
\chi (g_i)=\prod (\det g_i)^{\theta_i}
\end{equation*}
for some collection of integers $\theta_i$. This collection can be used to define a `$\theta$-functional', the value of which at a arbitrary representation $M=(W_i, \varphi)$ is
\begin{equation*}
\theta (M)=\sum \theta_i \dim W_i
\end{equation*}
In choosing such character to define semistability, the condition that the character be trivial on the kernel of the representation means in this case that $\sum \theta_i n_i=0$, or in other words that $\theta (M)=0$ for every representation with the chosen dimension vector $\mathbf{n}$.

A \emph{subrepresentation} of the representation $(V,\varphi)$ is a representation determined by a subspace $W\subset V$ that is $\varphi$-stable in the sense that $\varphi(W)\subset W$, i.e., the pair $(W,\varphi)$. The following is the result we're working toward., and is due to King \cite{king}.
\begin{proposition}
The representation $M$ is $\chi$-semistable if and only if for any non-trivial subrepresentation $M'$ we have $\theta (M')\leq 0$; it is stable is strict inequality always applies.
\end{proposition}
\begin{proof}
Given a one-parameter subgroup $\lambda$, we know that the limit exists if anf only if $\varphi \in \mathfrak{p}_G(\lambda)$. Concretely this means that $\lambda$ induces a flag $0\neq V^1 \subset ... \subset V^l=V$ of $V$ as a \emph{graded} vector space, determined by its eigenvalues. The condition on $\varphi$ is then that $\varphi$ restrict to each $V^j$, so that the pairs $M_j:=(V_j, \varphi)$ are well-defined representations. A computation shows then that
\begin{equation*}
\langle \chi, \lambda \rangle = \sum \theta (M_j)
\end{equation*}
This shows half of the theorem using Hilbert-Mumford. The other half comes from just considering the one-term flag for each subrepresentation.
\end{proof}

Suppose we are given an arbitrary collection of integers $\theta_i$, which also induces a linear functional $\theta$ on $\mathbb{Z}^I$; let also $\dim$ define the functional $(n_i)\mapsto \sum n_i$. Then, the functional $\theta'= (\dim V) \theta + \theta (\mathbf{n})\dim$ clearly is integral, and satisfies $\theta'(\mathbf{n})=0$. In other words, it obbeys the condition on characters for the existence of semistable points. The condition on the proposition is now that for a subrepresentation $M'$,
\begin{equation*}
\theta'(M')=(\dim V)\theta (M')+\theta (\mathbf{n})\dim M'\leq 0
\end{equation*}
If we define the \emph{slope} of a representation $M=(V,\varphi)$ with dimension vector $\mathbf{n}$ as
\begin{equation*}
s(M)=\frac{\theta (\mathbf{n})}{\dim \mathbf{n}}
\end{equation*}
we get an immediate corollary.
\begin{corollary}
A representation is semistable if and only if $s(M')\leq s(M)$ for every non-trivial subrepresentation $M'\subsetneq M$; it is stable if strict inequality always applies.
\end{corollary}

In general, distinct collections $\theta_i$ don't necessarily yield different semistability conditions in this way. In fact, semistability is invariant under multiplication of $\theta$ by integers, and sums of integer times the $\dim$-functional.

A consequence of this result is the characterization of polystable objects in terms of so-called \emph{Jordan-H\"older filtrations}. First note that given a subrepresentation $M'$ of $M$ as above, one can define the quotient representation $M'':=M/M'$ by taking the quotient of the total spaces, and noting that since $\varphi$ restricts to $M'$, it also factors through to the quotient. A Jordan-H\"older filtration to $M$ is then a filtration 
\begin{equation*}
0\neq M_1\subsetneq ... \subsetneq M_n=M
\end{equation*}
such that the successive quotients $M_i/M_{i-1}$ is stable. This filtration an be inductively constructed as follows: we may assume that the representation is strictly semistable, and we pick a minimal dimensional subrepresentation $M'$ such that $s(M')=s(M)$, and set $M_1=M'$; this subrepresentation is necessarily stable. If $M/M_1$ is not itself stable, we repeat the process.

The Jordan-H\"older filtration is not unique, but its associated graded object is so up to isomorphism. Instead of proving this directly, we will show that this in nothing else but an intepretation of Proposition \ref{jhgen} in terms of flags.
\begin{proposition}
The graded Jordan-H\"older objects coincide with the polystable representatives in Proposition \ref{jhgen}.
\end{proposition}
\begin{proof}
The result follows from a careful comparison of the procedure just described with the inductive proof of that proposition. In particular, we need to understand the inductive step in terms of subrepresentations. We note that the maximal parabolics are those fixing a minimal flag, i.e., those with only one non-trivial step $0\neq M'\subset M$. From the computation above of the Hilbert-Mumford pairing for a filtration, we see that such parabolic is admissible if and only if $s(M')=0$. This is precisely the inductive step in constructing a Jordan-H\"older filtration.
\end{proof}

An analogous study can be made for the instability type of the representation, in terms of the \emph{Harder-Narasimhan filtration}. This is the \emph{unique} filtration
\begin{equation}\label{hnfiltration}
0\neq M_1\subsetneq ... \subsetneq M_n=M
\end{equation}
such that the successive quotients $N_i:=M_i/M_{i-1}$ are semistable, and $s(N_1)> s(N_2)>...> s(N_n)$. We call the vector $(s(N_1),..., s(N_n))$ the \emph{Harder-Narasimhan type} of the representation.

\begin{proposition}\label{hnclassical}
Each most destabilizing conjugacy class of OPS determines a unique Harder-Narasimahn type for which the Hesselink stratum $S_{[\beta]}$ of the class is precisely the set of all representations of that type. Further, each blade $S_\beta$ is determined by further specifying a specific filtration of the total space (the other possible ones are conjugate.) Finally, the retraction $Z_\beta$ of $S_\beta$ by Bialinicky-Birula is precisely the set of graded objects for such types with fixed filtration.
\end{proposition}

If we use the coincidence of the Morse and Hesselink indices for the strata, a proof follows from Proposition 3.10 in \cite{hw}, and it was carried out in \cite{hoskins} and \cite{zamora}. We will obtain an alternative proof by showing that the conditions on the filtration imply the conditions on the parabolic in Theorem \ref{instparalg}.

\begin{proof}
Let $P$ be the parabolic subgroup corresponding to the filtration (\ref{hnfiltration}). The fact that the representation factors through that filtration is equivalent to the fact that $\varphi \in \mathfrak{p}$, and the condition on the semistability of the successive quotients is equivalent to the semistability of the projection $p(\varphi)$ to Levi subalgebra $\mathfrak{l}$. We have then to interpret the condition on the slopes. For some strictly dominant OPS to have the Hilbert-Mumford pairing with the character to be negative, is is necessary that some simple weight also have, so we may assume the one-parameter subgroup is defined by such. Now, the simple weight will induce a subfiltration of (\ref{hnfiltration}), i.e., for some step $M_j$ in that filtration, the filtration of the simple weight $\alpha$ is
\begin{equation*}
0\neq M_j \subsetneq M
\end{equation*}
Futher, the pairing $\langle \chi, -\alpha \rangle=\theta(M)\dim (M_j) - \dim (M) \theta(M_j)$. It follows that $s(N_j)\leq s(M)$, which contradicts the properties of the Harder-Narasimhan filtration.
\end{proof}

We finish the algebraic discussion by remarking that the slice theorem for the case of classical quivers was worked out by LeBruyn-Procesi \cite{lp}, and it is extremely explicit and computational.

\subsection{Morse Theory}
The results that follow are all due to Harada-Wilkin \cite{hw}. We highly recommend that paper not only for the details for these results, but also for a more details on our approach here. We will first compute the level set of the moment map following King \cite{king}, using formula (\ref{affinemm}). Fix a total space $V=\bigoplus V_i$; we need to introduce a hermitian metric on $\Rep (Q,V)$, which can be easily done by picking a hermitian metric separately on each $V_i$, and then defining on each $\Hom (V_i,V_j)$ the metric $(\varphi, \psi)=\tr (\varphi \psi^*)$. This automatically determines a maximal compact of $\GL (\underline{i})$, and the infinitesimal of its Lie algebra is $\beta \cdot \varphi=\beta_{\ha} \varphi_\alpha -\varphi_\alpha \beta_{\ta}$. It is now a simple computation to show
\begin{equation*}
(\beta\cdot \varphi, \varphi)=\sum_\alpha \tr \left( \beta_{\ha} \varphi_\alpha -\varphi_\alpha \beta_{\ta} \right)= \sum_i \tr \left( \beta_i \left( \sum_{\ha =i}\varphi_\alpha \varphi_\alpha^* - \sum_{\ta=i}\varphi^*_\alpha\varphi_\alpha \right) \right)
\end{equation*}
The last expression is clearly the standard pairing with $\beta$, and so actually gives a formula for the moment map after identification of $\mathfrak{u}(\mathbf{i})$ with its dual. To obtain compatibility with the algebraic side, we know we have to shift this moment map by the derivative of the character. Since this character is determined by the choice of integers $\theta_i$, we actually obtain the equation
\begin{equation*}
\sum_{\ha =i}\varphi_\alpha \varphi_\alpha^* - \sum_{\ta=i}\varphi^*_\alpha\varphi_\alpha=\theta_i I_i
\end{equation*}
where $I_i$ is the identity on $V_i$.

Given a quiver representation $A$, the \emph{Harder-Narasimhan-Jordan-H\"older filtration} of $A$ is the double obtained by first finding the Harder-Narasimhan filtration of $A$, and then combining it with the Jordan-H\"older filtrations of each factor (which are by definition semistable;) we then speak of the HNJH object associated to $A$ to refeer to the graded object of this double filtration. Given our use of the flags associated to parabolic subgroups to interpret the poly- and instability of representations, the following shouldn't be too surprising.
\begin{proposition}[\cite{hw} Theorem 5.3]
Let $A$ be a quiver representation. Then, its limit point $A_\infty$ under the square of the moment map is isomorphic to its HNJH-graded object $A^{HNJH}$.
\end{proposition}
The proof of this within our framework is a straightforward application of Proposition \ref{criticalgeneralized} together with Kempf-Ness. The inductive formula in Corollary \ref{indcohom} is easily seen to correspond to formula (7.10) in \cite{hw} (and, as mentioned in that paper, also to Reineke's formula in \cite{reineke2} when all semistable points are stable.)

\section{Supermixed quivers}\label{symmetric}
The kind of results we obtained for classical, or "plain" quivers are an example of what we can do when one is dealing with classical groups. Here we'll see that a `geometric interpretation' can be given also to generalized quivers associated with non-degenerate quadratic forms.

\subsection{Symmetric and supermixed quivers}
We want to start with a few general remarks which contextualize the work that follows. The starting point is the theorem of Derksen-Weyman \cite{dw}, which classifies generalized quivers of type $Z$ for the orthogonal and symplectic groups. 
\begin{definition}
\begin{enumerate}
\item A symmetric quiver $(Q, \sigma)$ is a quiver $Q$ equipped with an involution $\sigma$ on the sets of vertices and arrows such that $\sigma t(\alpha)=h \sigma (\alpha)$, and vice-versa, and that if $\ta=\sigma \ha$, then $\alpha= \sigma (\alpha)$.

\item An orthogonal, resp. symplectic, representation $(V, C)$ is a representation $V$ of $Q$ that comes with a non-degenerate symmetric, resp. anti-symmetric, quadratic form $C$ on its total space $V_\Sigma=\bigoplus_{i \in Q_0} V_i$ which is zero on $V_i \times V_j$ if $j\neq \sigma (i)$, and such that
\begin{equation}\label{alternating}
C(\varphi_\alpha v, w)+C(v, \varphi_{\sigma (\alpha)} w)=0
\end{equation}
\end{enumerate}
\end{definition}
Note that a dimension vector for an orthogonal representation must have $n_i=n_{\sigma (i)}$; we say that such a dimension vector is `compatible.' The theorem is the following:
\begin{theorem}[Derksen-Weyman \cite{dw}]
Let $G=\Or (n,\mathbb{C})$ (resp. $\Sp (n, \mathbb{C})$.) Then, to every generalized $G$-quiver $\tilde{Q}$ with dimension vector we can associate a symmetric quiver $Q$ in such a way that the generalized quiver representations of $\tilde{Q}$ correspond bijectively to orthogonal (resp. symplectic) representations of $Q$. Conversely, every symmetric quiver with dimension vector determines a generalized $\Or (n,\mathbb{C})$-generalized quiver.
\end{theorem}
In fact, there is a unifying definition, that of \emph{supermixed quiver}, which allows for a mix of orthogonal and symplectic symmetries depending on some extra data chosen for the vertices. These supermixed quivers in fact arise in the study of symmetric quivers themselves, as we'll see below, since the slice theorem for orthogonal representations will naturally be encoded by a supermixed quiver, as shown in \cite{bocklandt}.
\begin{definition}
\begin{enumerate}
\item A \emph{supermixed quiver} $(Q, \sigma, \epsilon)$ is a symmetric quiver $(Q, \sigma)$ together with a sign map $\epsilon : I \cup A \to \{ \pm 1\}$ such that $\epsilon_i \epsilon_{\sigma (i)}=\epsilon_\alpha \epsilon_{\sa}=1$.
\item A \emph{supermixed representation} is a representation $(V, \varphi)$ of $Q$ together with a non-degenerate quadratic form $C$ with the conditions: (a) its restriction as a bilinear form on $V_i\times V_j$ is zero if $j\neq \sigma (i)$; (b) the restrictrion as a quadratic form on the sum $V_i\oplus V_{\sigma{(i)}}$ has $C(v,v')=\epsilon_iC(v',v)$; and (c) $C(\varphi_\alpha v, v') + C(v, \varphi_{\sa}v')=0$.
\end{enumerate}
\end{definition}
Note that if $\epsilon$ is identically $1$ (resp. $-1$,) the we recover the orthogonal (resp. symplectic) representations above. The original reference for this is Zubkov \cite{zubkov1} \cite{zubkov2}, where the invariants for such quivers are computed.

\subsection{Stability of supermixed representations}

Since there is, up to isomorphism, a unique finite dimensional vector space in each dimension, and since over the complex numbers the (anti)symmetry uniquely determines the quadratic form, we may as well, in discussing supermixed quivers, fix a total space $V$ and the quadratic form $C$. Denote by $\Rep (Q,V,C)$ the space of supermixed representations; since every such representation is in particular a representation of $Q$, there is a `forgetful map'
\begin{equation*}
f:\Rep (Q,V,C) \to \Rep (Q,V)
\end{equation*}
which is clearly injective. Indeed, we can indetify the first as a subspace of the second explicitly as follows: the quadratic form $C$ induces an involution $*: \Rep (Q,V) \to \Rep (Q,V)$, namely transposition; the first space is then the $-1$ eigenspace of this involution. The symmetry group of a supermixed quiver can be found in the same way: there is also an adjoint map defined on $\GL (\mathbf{n})$, and the symmetry group is the group $\Or (\mathbf{n})$ of elements such that $g^*g=gg^*=1$ (we're here abusing notation, since the group in general is not a subgroup of the orthogonal group, but this avoids introducing new notation.) This group is isomorphic to a product
\begin{equation*}
\Or (\mathbf{n})\simeq \prod_{\begin{subarray}{c} i=\sigma (i)\\ \epsilon_i=1\end{subarray}} \Or (V_i) \times \prod_{\begin{subarray}{c} i=\sigma (i)\\ \epsilon_i=-1\end{subarray}} \Sp(V_i) \times \prod_{\begin{subarray}{c} [i]\\ i\neq \sigma (i)\end{subarray}} \GL(V_i)
\end{equation*}
\emph{where in the last product, we mean to take on factor for each orbit of $\sigma$,} and not for each $i$. Denote the set of indices in the first product by $O$, the second by $S$, and a fixed set of representatives for the orbits indexing the third product by $G$.

The map $f$ naturally induces a semistability condition on $\Rep (Q,V,C)$ by restriction of a character $\chi$ to $\Or (\mathbf{n})$. For such concordance of stability conditions, the map $f$ naturally descends to a map between the quotients of reprsentations. It is a result of Zubkov \cite{zubkov1} that for the trivial character on both, this natural map is actually a closed embedding.

However, a look at the above isomorphism of groups shows that these induced characters only give a small subset of possibilities. Instead, take integers $\theta_i$ for $i \in O\cup S\cup G$ with $\theta_i=0,1$ for $i\in O$, and $\theta_i=0$ for $i\in S$. Such vector of integers parametrizes the complete set of characters of $\Or (\mathbf{n})$. Since we also want to apply the symplectic machinery, we will always consider $\theta_i=0$ for $i\in O\cup S$; the condition on the kernel of the representation implies $\sum \theta_i n_i=0$.

We will now study the resulting stability properties in a way that is analogous to the case of classical quivers. The first thing to be done is to deduce a slope condition for stability. This can be done exactly like in the classical case: take a one-parameter subgroup $\lambda$ of $\Or (\mathbf{n})$, and consider the associated filtration of $V$. We define the theta functional just as above, except we take only one summand for each $i$ not fixed by $\sigma$, i.e.,
\begin{equation*}
\theta (M)=\sum_{i\in G} \theta_i n_i
\end{equation*}
since $\theta_i=0$ for $O\cup S$. King's computation straightforwadly extends to show that $\langle \chi, \lambda \rangle =\sum \theta (M_l)$, where $M_l$ are the steps in the filtration induced by $\lambda$. What we need now is to characterize the subrepresentations determined by parabolics of $\Or (\mathbf{n})$. Given a total space $V$, denote
\begin{equation*}
V'=\bigoplus_{i\in O\cup S} V_i \qquad V''=\bigoplus_{i\in G} (V_i\oplus V_{\sigma(i)})
\end{equation*}
Then, the parabolic of $\Or (\mathbf{n})$ induces a filtration $0\subset V_1\subset .... V_l \subset ... \subset V$ which is induced by filtrations on each of the vertices. In particular, it is a concatenation of filtrations on $V'$ and $V''$, and we have
\begin{itemize}
\item[---] The corresponding flag of $V'$ is isotropic in the sense of section \ref{parabolic};

\item[---] The filtration on each $V_i\oplus V_{\sigma (i)}$ is a `transposition,' in the sense that the filtration on $V_i$ is arbitrary, and the filtration on $V_{\sigma (i)}$ is dual filtration naturally induced by $C$.
\end{itemize}
Given a subrepresentation $M'=(W,\varphi)\subset M$ such that $W$ satisfies this conditions with respect to $V,C$, we will say it is an \emph{isotropic subrepresentation}, though again we are here appropriating terminology that is specific to the exclusively orthogonal or symplectic case. The result is then
\begin{proposition}
The representation $M$ is $\theta$-semistable if and only if for any non-trivial, isotropic subrepresentation $M'\subsetneq M$ we have $\theta (M')\leq 0$; it is stable if strict inequality always applies.
\end{proposition}
\begin{remark}
Since the $\theta$-functional only depends on half of the non-fixed vertices, it might be tempting to think that only those determine the stability of a representation. For example, one might want to extract the subquiver determined by those vertices and consider the induced representations of that new quiver by truncation. One should keep in mind, however, that whether a given subrepresentation of this new quiver is a subrepresentation of the old one is controlled also by the orthogonal and symplectic vertices, and so in fact they are always in the background conditioning the representations.
\end{remark}

Just as for classical quivers, one can -- and should -- relax the condition on the numbers $\theta_i$. For a representation $M$ with dimension vector $\mathbf{n}$, define then
\begin{equation*}
\dim' (M)=\sum_{i\in G} n_i
\end{equation*}
Since we want to keep $\theta_i=0$ for $i\in O\cup S$, we can only add multiples of $\dim'$, and not multiples of $\dim$. Therefore, the slope of the representation is defined as
\begin{equation*}
s(M)=\frac{\theta(M)}{\dim'(M)}
\end{equation*}
Repeating the argument for classical quivers for a collection of $\theta_i$, $i\in G$, arbitrary, we get
\begin{corollary}
A representation $M$ is semistable if and only if $s(M')\leq s(M)$ for all non-trivial subrepresentations $M'\subsetneq M$; it is stable if strict inequality always applies.
\end{corollary}

Using this slope condition, we can now formally define Jordan-H\"older objects and Harder-Narasimhan filtrations. To construct the first, suppose the representation $M$ is strictly semistable, and choose a minimally dimensional, non-trivial isotropic subrepresentation $M_1\subset M$; this subrepresentation determines a maximal parabolic stabilizing the flag
\begin{equation*}
0\neq M_1\subsetneq M_1^\perp \subsetneq M
\end{equation*}
(The last inclusion is strict since the quadratic form is non-degenerate.) The graded representation $W_l$ associated with this filtration is naturally a representation for some Levi subgroup $L_1$. We form the quotient $M_1^\perp/M_1$, which is a well-defined supermixed representation, and repeat the process, finding a chain of representations corresponding to a chain $L_1\supset L_2 \supset...$ of successively smaller Levis. This process must stop at some step $l$, for at some point $M_l$ is necessarily stable or of minimal rank. The associated representation $W_l$ must be stable for as a representation associated with $L_l$: for otherwise a destabilizing one-parameter subgroup would imply some $M_j$ is not stable. We conclude then
\begin{proposition}
The graded representation $W_l$ obtained by the inductive process above is precisely the Jordan-H\"older object for $M$.
\end{proposition}

We want now to characterize also the Hesselink strata in terms of filtrations; in other words, we want to find the Harder-Narasimhan object for a given representation. Assume $M$ is an unstable representation, and let $M_1$ an isotropic subrepresentation of maximal slope, and maximal dimension with that property. Again, this fits into a flag
\begin{equation*}
0\neq M_1\subsetneq M_1^\perp \subsetneq M
\end{equation*}
corresponding to some parabolic subgroup $P_1$. The associated graded object (i.e., the object corresponding to the projection to the Levi subalgebra) is
\begin{equation*}
M_{\mathrm{gr},1}=(M_1\oplus M_1^*)\oplus M_1^\perp /M_1
\end{equation*}
where recall that using the quadratic form we get an isomorphism $M_1^*=M/M_1^\perp$. This is a splitting as an orthogonal representation, since both $M_1\oplus M_1^*$ and $M_1^\perp/M_1$ are orthogonal representations. The condition on $M_1$ ensures that $M_1\oplus M_1^*$ is actually semistable. If $M_1^\perp/M_1$ is not, then we repeat the procedure. The result is a filtration
\begin{equation*}
0\neq M_1 \subsetneq ... M_l \subsetneq M_l^\perp \subsetneq ... \subsetneq M_1^\perp \subsetneq M
\end{equation*}
where we have $\mu(M_1)>...>\mu (M_l)$, and also that $M_k/M_{k-1}\oplus M_{k-1}^\perp/M_k^\perp$$k=1,...,l-1$ and $M_l^\perp/M_l$ are semistable orthogonal representations. In analogy with the plain case, we'll refer to this filtration as the Harder-Narasimhan filtration of $M$. Arguing as in Proposition \ref{hnclassical} we can prove the following proposition.
\begin{proposition}
Each most destabilizing conjugacy class of OPS determines a unique Harder-Narasimhan type for which the Hesselink stratum $S_{[\beta]}$ of the class is precisely the set of all representations of that type. Further, each blade $S_\beta$ is determined by further specifying a specific filtration of the total space (the other possible ones are conjugate.) Finally, the retraction $Z_\beta$ of $S_\beta$ by Bialinicky-Birula is precisely the set of graded objects for such types with fixed filtration.
\end{proposition}

\subsection{An example}

We will now apply the inductive formula we deduced above to particular examples of orthogonal representations of the symmetric quiver

\begin{center}
\begin{tikzcd}[column sep=large]
Q:& 1& 2\ar[l,"\alpha"] & 3 \ar[l,"\beta"]\ar[r,"\sigma(\beta)" ']\ar[loop,"\gamma",swap,looseness=4] & \sigma (2) \ar[r,"\sigma(\alpha)" '] & \sigma (1)
\end{tikzcd}
\end{center}

Here, vertex 3 and arrow $\gamma$ are fixed by the involution; we will fix the dimension vector $d=(1,1,n)$. The choice of a stability condition is the choice of two integers $\theta_1$ and $\theta_2$. Stability of representations depends principally on the relative value of these parameters.
\begin{itemize}
\item $\theta_1=\theta_2$: This is the case of the trivial character, and the inexistence of instability renders our formula quite trivial. However, generators for the coordinate ring of the moduli of representations have been computed by Zubkov \cite{zubkov1} \cite{zubkov2} and Serman \cite{serman}.

\item $\theta_1<\theta_2$: Since $d_1=d_2=1$, a subrepresentation $M'\subset M$ is destabilizing if and only if it has dimension vector $d'=(0,1,n')$, and this is only possible if $\alpha=0$. The most destabilizing $E$ is then determined by the maximally isotropic $E_3\subset V_3$ with $\gamma (E_3)\subset E_3$. It is then a semistable plain representation of the quiver
\begin{center}
\begin{tikzcd}[column sep=large]
Q':&  2 & 3 \ar[l,"\beta"]\ar[loop,"\gamma",swap,looseness=4]
\end{tikzcd}
\end{center}
The induced stability condition is trivial: its subrepresentation can only have slope zero or $\mu(E')=\theta_2>0$. denote $n_1=\dim E_3$, and $n_2=n-2n_1$; then, in the Harder-Narasimhan splitting $M=E\oplus E^*\oplus D$, $n_2=\dim D_3$. Further, $D$ is an orthogonal representation of $Q'$ above, satsifying two extra conditions: first, the map $\beta_D: D_3\to D_2$ is non-zero if and only if $\beta (E_3)=0$; second, the map $\gamma_D$ cannot fix any isotropic subspace (by definition of $E_3$,) and this is equivalent to the representation $D$ actually being orthogonally stable for the trivial character. In other words, if we let
\begin{center}
\begin{tikzcd}[column sep=large]
Q'':&  3 \ar[loop,"\gamma",swap,looseness=4] 
\end{tikzcd}
\end{center}
then each critical Hesselink stratum $Z_{\beta}$ will be either of the form
\begin{equation*} 
Z_1(n_1,n_2):=\Rep (Q', 1,n_1)\oplus \Rep_0^{\mathrm{st}} (Q'', n_2)
\end{equation*}
or
\begin{equation*}
Z_2(n_1,n_2):=\Rep (Q'',n_1)\oplus \Rep_0^{\mathrm{st}}(Q',1,n_2)
\end{equation*}
These spaces are in fact the same, but their different notation denotes also a different action of the Levi. The corresponding Levi is just $L(n_1,n_2)=\mathbb{C}^*\times \GL(n_1)\times \Or (n_2)$. The codimension of $S_{[\beta]}$ is
\begin{align*}
d(n_1,n_2)=\frac{n^2+n+2}{2}-n_1^2-n_1-\frac{n_2(n_2-1)}{2}-1
\end{align*}

Conversely, every such combination for two integers $n_1$ and $n_2$ with $2n_1+n_2=n$ give a Hesselink stratum. To apply our inductive formula, we must first choose a unique representative in the conjugacy class $S_{[\beta]}$. This corresponds to the choice of a unique isotropic $E_3\subset V_3$ up to conjugation, and these are indexed precisely by the dimension of $E_3$. Therefore, in our inductive formula we will have precisely one summand for each combination $(n_1,n_2)$, i.e.,
\begin{align*}
P_t^{\Or (1,1,n)}(\Rep_0(Q, 1,1,n))=P_t(B\Or (1,1,n))+\sum_{\begin{subarray}{c} n_1\\ 2n_1=n\end{subarray}}&t^{2d(n_1, n-2n_1)}P_t^{L(n_1,n_2)}(Z_1(n_1,n_2))+\\ &+\sum_{\begin{subarray}{c} n_1\\ 2n_1=n\end{subarray}}t^{2d(n_1, n-2n_1)}P_t^{L(n_1,n_2)}(Z_2(n_1,n_2))
\end{align*}
Finally, we note that since the stability conditions on $Z_1$ and $Z_2$ are trivial, only the cycle part of the quiver contributes to its equivariant cohomology. In other words, if we define
\begin{align*}
Z(n_1)&:=\Rep (Q'',n_1)\oplus \Rep_0^{\mathrm{st}}(Q'',n-2n_1)\\
L(n_1)&:=\GL(n_1) \times \Or (n-2n_1)
\end{align*}
our formula reduces to
\begin{align*}
P_t^{\Or (1,1,n)}(\Rep_0(Q, 1,1,n))=P_t(B\Or (1,1,n))+2\sum_{\begin{subarray}{c} n_1\\ 2n_1=n\end{subarray}}&t^{2d(n_1, n-2n_1)}P_t^{L(n_1)}(Z(n_1))
\end{align*}
We have therefore reduced the induction to the computation of the equivariant cohomology of previously known cases. In fact, note that these cases are all of the adjoint representation proper.

\item $\theta_1>\theta_2$: Here a destabilizing representation must have dimension vector $d'=(1,0,n_1)$, which implies that the restriction of $\alpha$ and $\beta$ are both zero. Therefore, the most destabilizing representation is just the choice of a maximal isotropic $E_3$ fixed by $\gamma$, and so this is parametrized by representations of $Q''$ with dimension vector $n_1$. The corresponding $D$ in the Harder-Narasimhan splitting is again a stable representation of $Q'$ with dimension vector $(1,n_2)$. The critical Hesselink stratum is then
\begin{equation*}
Z(n_1,n_2):=\Rep (Q'',n_1)\oplus \Rep_0^{\mathrm{st}}(Q',1, n_2)
\end{equation*}
The induced stability conditions are again trivial. We can proceed as above to reduce in this way the inductive formula to known cases of the one-loop quiver.

\end{itemize}

\bibliographystyle{plain} 
\bibliography{references}

\end{document}